\documentclass[a4paper]{article}

\usepackage{amsmath,amssymb,graphicx} 
\usepackage{amsthm}
\usepackage[margin=1in]{geometry} 
\usepackage{enumerate}
\usepackage{tikz}
\usepackage{hyperref}
\usepackage[flushleft]{threeparttable}
\usepackage{array,booktabs,makecell}

\newtheorem{theorem}{Theorem}[section]
\newtheorem{definition}{Definition}[section]
\newtheorem{notation}{Notation}

\newtheorem{proposition}{Proposition}[section]
\newtheorem{remark}{Remark}[section]

\newtheorem{lemma}{Lemma}[section]
\newtheorem*{lemma*}{Lemma}
\newtheorem{corollary}{Corollary}[section]
\newtheorem{assumption}{Assumption}[section]

\usepackage{todonotes,caption,subcaption}

\begin{document}

\nocite{*} 

\title{Long time asymptotics of mixed-type Kimura diffusions}

\author{Guillaume Bal\thanks{Departments of Statistics and Mathematics and CCAM, University of Chicago, Chicago, IL 60637; {\tt guillaumebal@uchicago.edu}} \and Binglu Chen\thanks{Department of Mathematics, University of Chicago, Chicago, IL 60637; {\tt blchen@uchicago.edu}} \and Zhongjian Wang\thanks{Departments of Statistics and CCAM, University of Chicago, Chicago, IL 60637; {\tt zhongjian@uchicago.edu}}}

\date{} 

\maketitle
\tableofcontents
  
\begin{abstract}
    This paper concerns the long-time asymptotics of diffusions with degenerate coefficients at the domain's boundary. Degenerate diffusion operators with mixed linear and quadratic degeneracies find applications in the analysis of asymmetric transport at edges separating topological insulators. In one space dimension, we characterize all possible invariant measures for such a class of operators and in all cases show exponential convergence of the Green's kernel to such invariant measures. We generalize the results to a class of two-dimensional operators including those used in the analysis of topological insulators. Several numerical simulations illustrate our theoretical findings.
    \\
    \noindent\textbf{Key Words:} Degenerate Diffusion Operators; Invariant Measure; Long Time Asymptotics; Fredholm Index; Lyapunov Functions.\\
    \noindent\textbf{MSC}: 	
    35K65,  	
    60J70,  	
    47D06,      
    37C40.      
\end{abstract}

\section{Introduction}
	This paper analyzes the long time behavior of diffusion processes with infinitesimal generator given by a mixed type Kimura operator $L$ on one-dimensional and two-dimensional manifolds with corners. 
	
	In two-space dimensions, the mixed type Kimura operator $L$ acts on functions defined on a manifold with corner $P$. A boundary point $p\in bP$ has a relatively open neighborhood $V$ that is homomorphic to either $\mathbb{R}_+^2$ or $\mathbb{R}_+\times\mathbb{R}$, where we assume that $p$ maps to $(0,0)$. We say that $p$ is a corner in the first case and an edge point in the second case. More details are presented in Section \ref{sec:2d}. 
	
	In this paper, $L$ is a degenerate second-order differential operator such that for each edge, the coefficients of the normal part of the second-order term vanish to order one or two. For an edge point $p$, when written in an adapted system of local coordinates on $\mathbb{R}_+\times\mathbb{R}$, $L$ takes the form:
	\begin{gather}\label{form_edge}
		L=a(x,y)x^m\partial_{xx}+b(x,y)x^{m-1}\partial_{xy}+c(x,y)\partial_{yy}+d(x,y)x^{m-1}\partial_x+e(x,y)\partial_y,
	\end{gather}
	where we assume that $a,b,c,d,e$ are smooth functions and that $a(x,y)>0$ and $c(x,y)>0$. 
	Also $ m\in \{1,2\}$. When $m=1$, the edge $x=0$ is of Kimura type in that the coefficients vanish linearly towards it. When $m=2$, the edge $x=0$ is of quadratic type as the coefficients now vanish quadratically.
	
	For a corner $p$, as an intersection point of two edges, $L$ has the following normal form:{\small
	\begin{gather}\label{form_corner}
		L=a(x,y)x^m\partial_{xx}+b(x,y)x^{m-1}y^{n-1}\partial_{xy}+c(x,y)y^n\partial_{yy}+d(x,y)x^{m-1}\partial_x+e(x,y)y^{n-1}\partial_y,
	\end{gather}}
	when written in an adapted system of local coordinates on $\mathbb{R}_+^2$, where $a(x,y)>0$, $c(x,y)>0$,
	and $m,n\in \{1,2\}$.
	
	Associated to the operator $L$ is a $C^0$ semigroup $\mathcal{Q}_t=e^{tL}$ solution operator of the Cauchy problem
	\begin{align*}
		\partial_t u = Lu
	\end{align*}
	with initial conditions $u(x,0)=f(x)$ at $t=0$. The operator $e^{tL}$ and some of its properties are presented in detail in \cite{chen2022mixed}. The main objective of this paper is to analyze the long time behavior of $e^{tL}$, and in particular convergence to appropriately defined invariant measures.
	
	It turns out that the number of possible invariant measures and their type (absolutely continuous with respect to one-dimensional or two-dimensional Lebesgue measures or not) strongly depend on the structure of the coefficients $(a,b,c,d,e)$. We thus distinguish the different boundary types that influence the long time asymptotics of transition probabilities. 
	\begin{definition}
		A Kimura edge $E$ is called a tangent (Kimura) edge when $d(0,y)=0$ and a transverse (Kimura) edge when $d(0,y)>0$.
		
		A quadratic edge $E$ is called a tangent (quadratic) edge when $\frac{d(0,y)}{a(0,y)}<1$, a transverse (quadratic) edge when $\frac{d(0,y)}{a(0,y)}>1$, and a neutral (quadratic) edge when  $\frac{d(0,y)}{a(0,y)}=1$.
	\end{definition} 
	We assume that:
	\begin{assumption}
		Every edge is either tangent, transverse, or neutral.
	\end{assumption}
	Note that we do not consider the setting with $d(0,y)<0$ on a Kimura edge. In such a situation, diffusive particles pushed by the drift term $d(0,y)<0$ have a positive probability of escaping the domain $P$. We would then need to augment the diffusion operator with appropriate boundary conditions. 
	
	The long-time analysis in two-dimensions for a class of operators including a specific example of interest in the field of topological insulators \cite{topological} is given in section \ref{sec:2d}. In the application in \cite{topological}, the degenerate diffusion equation describes reflection coefficients of wavefields propagating in heterogeneous media, which model the separation between topological insulators; see \cite{bal2019continuous,bal2022asymmetric}.
	Before addressing two-dimensional operators, we consider the simpler one-dimensional setting, where we need to consider ten different scenarios, listed in table \ref{tab:1d} below, depending on the form of the generator $L$ at the domain's boundary. These cases are analyzed in detail in sections \ref{sec_semigroup}-\ref{sec:cv1d}. 
	
	To describe all invariant measure on the interval $[0,1]$ in one dimensional, consider
	\begin{equation}\label{operator1D}
		L=a(x)x^{m(0)}(1-x)^{m(1)}\frac{d^2}{dx^2}+b(x)x^{m(0)-1}(1-x)^{m(1)-1}\frac{d}{dx},
	\end{equation}
	with ${a}(x), {b}(x)\in C^\infty([0,1])$.
	
	For $i=0,1$,  we say that $x=i$ is of Kimura type when $m(i)=1$ and of quadratic type when $m(i)=2$. When $x=i$ is of Kimura type, we assume that the vector field $b(x)\frac{d}{dx}$ is inward pointing at $x=i$. For brevity, we use $\widetilde{a}$ and $\widetilde{b}$ to denote 
	\begin{align}\label{tilde_notation}
		\widetilde{a}(x)=a(x)x^{m(0)}(1-x)^{m(1)}, \quad \widetilde{b}(x)=b(x)x^{m(0)-1}(1-x)^{m(1)-1}.
	\end{align}
	\begin{definition}
		When $x=0\ (1\ \text{resp.}) $ is a Kimura endpoint, we say it is a tangent point if $b(0)=0\ (b(1)=0\ \text{resp.})$ and a transverse point if $b(0)>0\ (b(1)<0\ \text{resp.})$.
		
		When $x=0\ (1\ \text{resp.})$ is a quadratic endpoint, we say it is a tangent point if $\frac{{b}(0)}{{a}(0)}<1\ (\frac{{b}(1)}{{a}(1)}>-1\ \text{resp.})$, a transverse point if  $\frac{{b}(0)}{{a}(0)}>1\ (\frac{{b}(1)}{{a}(1)}<-1\ \text{resp.})$, and a neutral point if  $\frac{{b}(0)}{{a}(0)}=1\ (\frac{{b}(1)}{{a}(1)}=-1\ \text{resp.})$.
	\end{definition}

	The quadratic endpoint and the tangent Kimura endpoint are sticky boundary points in the sense that the Dirac measure supported on them is an invariant measure. When both endpoints are transverse, there is another invariant measure $\mu$ with full support on the whole interval. By computing the index of $L$ on an appropriate H\"older space, we characterize the kernel space of $\overline{L}^*$ composed of invariant measures for the diffusion $L$.
	
	In both cases, starting from a point in $P$, the corresponding transition probability of the diffusion converges to the invariant measure at an exponential rate. In cases with at least one tangent boundary point, we consider a functional space of functions that vanish at the tangent boundary points and show that $L$ has a spectral gap on such a space. In the absence of tangent boundary points, we prove that the invariant measure $\mu$ satisfies an appropriate Poincar\'e inequality so that $L$ also admits a spectral gap in $L^2(\mu)$. Our main convergence results for $\mathcal{Q}_t=e^{tL}$, whose properties are described in Theorem \ref{semigroup}, are summarized in Theorems \ref{wasser_tant} and \ref{q_t:l^2} below. 
	
	\medskip

	In two space dimensions, we do not consider all possible invariant measures as a function of the nature of the  drift terms $d$ and $e$ in the vicinity of edges or corners. Instead, we restrict ourselves to the following case:
	\begin{assumption}\label{assu:2d}
		For $L$ on a 2 dimensional compact manifold with corners $P$, there is exactly one tangent edge $H$, and when restricted to $H$, $L|_H$ is transverse to both boundary points.
	\end{assumption}
	This case involves exactly one tangent edge with two transverse boundary points so that, applying results from the one-dimensional case, we find that $L$ has a unique invariant measure $\mu$ fully supported on (the one-dimensional edge) $H$. Starting from any point $p$ not on the quadratic edge, we show in Theorem \ref{col:wasserstein} that the transition probability converges to $\mu$ at an exponential rate in the Wasserstein distance sense. The main tool used in the convergence is the construction of a Lyapunov function in Theorem \ref{thm:lyapunov}.
	
	The setting of $P$ a triangle with two transverse Kimura edges while the third edge is quadratic with transverse endpoints as described in Assumption \ref{assu:2d} finds applications in the analysis of the asymmetric transport observed at an edge separating topological insulators \cite{topological}; see Remark \ref{rem:topological} below. The corresponding one-dimensional version with one transverse Kimura point and one tangent quadratic point (see entry $(1,4)$ in Table \ref{tab:1d} below) also appears in the analysis of reflection coefficients for one-dimensional wave equations with random coefficients \cite{fouque2007wave}.
	
	
	\medskip
	
	There is a large literature on the analysis of the long-time behavior of $e^{tL}$ when $L$ is non-degenerate and when $L$ is of Kimura type. In the latter case,  $L$ is the generalized Kimura operator studied in \cite{epstein2013degenerate}. In that work, $L$ is analyzed on a weighted H\"older space, denoted by $L_\gamma: C^{0,2+\gamma}(P)\to C^{0,\gamma}(P)$. $L_\gamma$ is then Fredholm of index zero, which is used to characterize the nullspace of an adjoint operator and show that the non-zero spectrum lies in a half plane $\text{Re}\ \mu<\eta<0$ so that for $f\in C^{0,\gamma}(P)$, $e^{tL}f$ converges to a stable limiting solution at an exponential rate.
	
	
	However, in the presence of quadratic edge/point, $L_\gamma$ is not Fredholm. The reason is that near such quadratic edges or points, the operator may be modeled by an elliptic (non-degenerate) operator on an infinite domain (with thus continuous spectrum in the vicinity of the origin). We thus need another approach that builds on the following previous works.  In \cite{arendt1986one}, the growth bound of a strongly continuous positive semigroup is associated with the Lyapunov function. In \cite{bakry1985diffusions}, for a time continuous Markov process admitting a (unique) ergodic invariant measure, the rate of convergence to equilibrium is studied using a weaker version of a Lyapunov function called a $\phi$-Lyapunov function and an appropriate Poincar\'e inequality.

	
	\medskip
	An outline of the rest of this paper is as follows. The semigroup $e^{tL}$ is analyzed in section \ref{sec_semigroup} in the one-dimensional case. The space of invariant measures associated with a one-dimensional diffusion $L$, which depends on the structure of the drift term at the two boundary points, is constructed in section \ref{sec:im1d}; see Table \ref{tab:1d} for a summary. The exponential convergence of the kernel of $e^{tL}$ (the Green's function) to an appropriate invariant measure over long times is demonstrated in section \ref{sec:cv1d}. 
	
	The operator $e^{tL}$ in the two-dimensional setting is analyzed in \cite{chen2022mixed}. For the class of operators satisfying Assumption \ref{assu:2d}, the construction of the (unique normalized) invariant measure and the exponential convergence of the kernel of $e^{tL}$ to this measure in the Wasserstein sense are given in section \ref{sec:2d}. 
	Numerical simulations of stochastic differential equations presented in section \ref{sec:num} illustrate the theoretical convergence results obtained in dimensions one and two.

	\section{The $C^0$ Semigroup in one-space dimension}\label{sec_semigroup}
	Let $L$ be the one-dimensional mixed-type Kimura operator on $[0,1]$ given in \eqref{operator1D}.
	Let $\mathcal{Q}_t=e^{tL}$ be the solution operator of the Cauchy problem for the generator $L$ and denote by $q_t(x,y)$ its kernel. Its main properties are summarized in the following result:
	\begin{theorem}\label{semigroup}
		The operator $\mathcal{Q}_{t}$ defines a positivity preserving semigroup on $C^0([0,1])$. For $f\in C^0([0,1])$, the function $u(x,t)=\mathcal{Q}_tf(x)$ solves the Cauchy Problem for $L$ with initial condition $f(x)$ in the sense that 
		\begin{gather}\label{3}
			\underset{t \to 0+}{\lim}||\mathcal{Q}_tf-f||_{\mathcal{C}^{0}}=0.
		\end{gather}
	\end{theorem}
	\begin{proof}
		If both endpoints are of Kimura type, $L$ is the 1D Kimura operator. By \cite[Sec 9, Theorem 1]{Epstein2010WrightFisherDI}, $\mathcal{Q}_t$ defines a positive and strongly continuous semi-group on $C^m([0,1])$ for $m\in\mathbb{N}$. If both endpoints are of quadratic type, we can regard $L$ as a uniform parabolic operator on $\mathbb{R}$ (using, e.g., a change of variables $x=e^z$ in the vicinity of $x=0$ as we do below), which is well studied (see, e.g., \cite{Fri}).  
		
		For the remaining case, we might as well assume that $x=0$ is a Kimura endpoint and $x=1$ a quadratic endpoint. We intend to build the global solution out of local solution near the boundary. Let $([0,1-\eta], \phi_0), ([\eta,1], \phi_1)$ for some $0<\eta<\frac{1}{4}$ small be the coordinate charts so that pulling back $L$ to these coordinate charts gives two local operators
		\begin{gather*}
			L_0=x\partial^2_x+b_0\partial_x+xc(x)\partial_x, \quad x\in [0,\phi_0(1-\eta)),\\ L_1=\partial_z^2+d(z)\partial_z, \quad z\in (\phi_1(\eta), \infty).
		\end{gather*}
		We extend these two local operators to the whole sample space
		\begin{gather*}
			\widetilde{L}_0=x\partial^2_x+b_0\partial_x+xc(x)\varphi_0(x)\partial_x,\quad \widetilde{L}_1=\partial_z^2+d(z)\varphi_1(z)\partial_z
		\end{gather*}
		where $\varphi_0(x)$ is a smooth cutoff function so that
		\begin{gather*}
			\varphi_0(x)=\begin{cases}
				\begin{array}{cc}
					1 & \text{for}\  x\in [0,\phi_0(1-2\eta)]\\
					0 & \text{for}\  x>\phi_0(1-\eta),
			\end{array}\end{cases}
			\varphi_1(z)=\begin{cases}
				\begin{array}{cc}
					1 & \text{for}\  z\in [\phi_1(2\eta), \infty)\\
					0 & \text{for}\  z<\phi_1(\eta).
			\end{array}\end{cases}
		\end{gather*}
		Let $\widetilde{Q}^0_t, \widetilde{Q}^1_t$ be the solution operators of $\widetilde{L}_0, \widetilde{L}_1$ and denote their kernels by $\widetilde{q}^0_t, \widetilde{q}^1_t$ respectively. Define smooth cutoff functions $0\leq\chi, \psi_0, \psi_1\leq 1$ so that 
		\begin{gather}\label{choice}
			\text{supp}\psi_0\subset [0,1-2\eta],\ \text{supp}\psi_1\subset [2\eta,1],\  \psi_0|_{\text{supp}\chi}\equiv 1,\  \psi_1|_{\text{supp}(1-\chi)}\equiv 1.
		\end{gather}
		
		Given $f\in C^0([0, 1])$ and $g\in C^0([0,1]\times [0,T])$, set the homogeneous and inhomogeneous solution operator as
		\begin{gather*}
			\widetilde{Q}_tf=\psi_0\widetilde{Q}^0_t[\chi f]+\psi_1\widetilde{Q}^1_t[(1-\chi)f],\\
			A_tg=\int_0^t\widetilde{Q}_{t-s}g(s)ds.
		\end{gather*}
		Then 
		\begin{gather*}
			(\partial_t-L)\widetilde{Q}_tf=E^0_tf:=[\psi_0,L]\widetilde{Q}^0_t[\chi f]+[\psi_1,L]\widetilde{Q}^1_t[(1-\chi)f],\\
			(\partial_t-L)A_tg=(Id-E_t)g:=g-[\psi_0,L]A^0_t[\chi g]-[\psi_1,L]A^1_t[(1-\chi)g].
		\end{gather*}
		
		Our choice of $\chi, \psi_0, \psi_1$ \eqref{choice} ensures that $\text{dist}(\text{supp}[\psi_0, L], \text{supp}\chi)>0$, $ \text{dist}(\text{supp}[\psi_1, L], \text{supp}(1-\chi))>0$, which ensures that $E^0_t, E_t$ are bounded operators with operator norms bounded by $O(e^{-\frac{c}{t}})$ for some constant $c>0$ as $t\to 0^+$. Hence for $T>0$ small enough, there exists an inverse $(Id-E_t)^{-1}$, which can be expressed as a convergent Neumann series in the operator norm topology of $C^0([0,1]\times [0,T])$. Finally we can express the solution operator by 
		\begin{gather*}
			\mathcal{Q}_tf=\widetilde{Q}_tf-A_t(Id-E_t)^{-1}E^0_tf.
		\end{gather*}
		
		Since both $\widetilde{Q}^0_t, \widetilde{Q}^1_t$ are strongly continuous, then so is $\widetilde{Q}_t$: \[\underset{t\to 0}{\lim}\  ||\widetilde{Q}_tf-f||_{C^0}=0.\] 
		As $(Id-E_t)^{-1}E^0_t$ is a bounded map from $C^0([0,1]$ to $C^0([0,1]\times [0,T])$, we have
		\[||A_t(Id-E_t)^{-1}E^0_t||_{C^0([0,1])\to (C^0([0,1]), t)}=o(t).\] 
		Therefore \eqref{3} holds. Let
		$\widetilde{q}_t, h_t$ be the heat kernels of $\widetilde{Q}_t, (Id-E_t)^{-1}E^0_t$. We express the heat kernel of $\mathcal{Q}_t$ as
		\begin{gather*}
			q_t(x,y)=\widetilde{q}_t(x,y)-\int_0^t\int_0^1\widetilde{q}_{t-s}(x,z)h_s(z,y)dzds.
		\end{gather*}
	\end{proof}



	\section{Invariant Measures in dimension one}\label{sec:im1d}
	In this section, we aim to find all invariant measures of $L$ in spatial dimension one. For convenience of computation, in this section we first choose a global coordinate $\phi$ so that $(W_{0},\phi),(W_{1},\phi)$ is a cover of $[0,1/3], [2/3,1]$ under which $L$ takes the following normal form:
	\begin{enumerate}
		\item In $([0,1/3],\phi)$, $L_{0}$ has two possible forms:
		\begin{gather*}
			L_{0}=x\partial_{x}^{2}+b(x)\partial_{x},\ \text{if 0 is a Kimura endpoint}\\ L_{0}=\partial_{z}^{2}+b(z)\partial_{z},\ \text{if 0 is a quadratic endpoint}
		\end{gather*}
		
		\item In $([2/3,1],\phi)$, $L_{1}$ has two possible forms:
		\begin{gather*}
			L_{1}=(1-x)\partial_{x}^{2}+b(x)\partial_{x}, \ \text{if 1 is a Kimura endpoint}\\ L_{1}=\partial_{z}^{2}+b(z)\partial_{z}, \ \text{if 1 is a quadratic endpoint}
		\end{gather*}
	\end{enumerate}
	where we use $b$ to denote the first-order term in all cases.
	\begin{notation}
		We call the global coordinate $\phi$ on $[0,1]$ {\em heat coordinates} if $L_0, L_1$ have forms above under $\phi$.
	\end{notation}
	Let
	\begin{gather}\label{b(z)}
		b_{\pm}=\underset{x\to 1,0}\lim b(x) \qquad \text{ or } \qquad  b_{\pm}=\underset{z\to\pm\infty}\lim b(z).
	\end{gather}
	A straightforward derivation shows that, if in the original coordinate, $L_{0}=x^2\partial_{x}^{2}+(b_{-}+1)x\partial_{x}$, then after turning to heat coordinates $x=e^z$, $L_0$ takes the form $\partial_{z}^{2}+b_-\partial_{z}$.
		\subsection{Functional settings and index associated to $L$}\label{holder_space}
		Our functional setting involves local H\"older spaces, which differ from the usual H\"older space (with $|\cdot|_{k+\gamma}$ to denote its norm) near the boundaries and are variations of those used in \cite{epstein2013degenerate}. In the following definition, we assume $c$ is a point away from the interval's boundaries.  
		\begin{enumerate}
			\item 
			For quadratic type boundaries, with $U=(-\infty,c]$,
			\begin{enumerate}
				\item  $f\in C^0(U)$ belongs to $\mathcal{D}^{\gamma}(U)$ if the function $f$ can be continuously extend to $z=-\infty$ and the following norm is finite:
				\begin{gather}
					{||f||_{\gamma,U}= |f|_{\gamma,U}}+\underset{z_1\leq z_2\leq c}{\sup}\big|\int_{z_1}^{z_2}fdz\big|, \text{when it is not neutral};\\
					{||f||_{\gamma,U}= |f|_{\gamma,U}}+\underset{z_1\leq z_2\leq c}{\sup}|\int_{z_1}^{z_2}fdz|+\underset{z_1\leq z_2\leq c}{\sup}|\int_{z_1}^{z_2}\int_{-\infty}^{z}fdxdz|\\, \text{when it is neutral}\nonumber;
				\end{gather}
				
				\item 
				$f\in C^1(U)$ belongs to $C^{1+\gamma}(U)$ if the functions
				$\partial_zf,f$ can be continuously extend to $z=-\infty$ and local $C^{1+\gamma}$ norm is finite:
				\begin{gather}
					||f||_{1+\gamma,U}= |f|_{1+\gamma,U}+||\partial_z f||_{\gamma,U};
				\end{gather}
				
				\item $f\in C^2(U)$ belongs to $C^{2+\gamma}(U)$ if the functions
				$\partial_z^2f, \partial_{z}f,f$ extend continuously to $z=-\infty$ and the local $C^{2+\gamma}$ norm is finite:
				\begin{gather}
					{||f||_{2+\gamma,U}= |f|_{2+\gamma,U}}+||\partial_z^2f||_{\gamma,U}.
				\end{gather}
			\end{enumerate}
			
			\item For Kimura type boundaries, with $U=(0,c]$,
			\begin{enumerate}
				\item $f\in C^0(U)$ belongs to $\mathcal{D}^\gamma(U)$ if the function $xf$ can be continuously extend to $x=0$, and the local $C^\gamma$ norm is finite: \begin{gather}
					{||f||_{\gamma,U}}= |xf|_{\gamma,U};
				\end{gather}
				
				\item $f\in C^0(\overline{U})\cap C^1(U)$ belongs to $C^{1+\gamma}(U)$ if the function $x\partial_xf$ can be continuously extend to $x=0$ and vanish, and the local $\mathcal{C}^{1+\gamma}$ norm is finite:
				\begin{align}
					{||f||_{1+\gamma,U}}=  |f|_{\gamma,U}+|x\partial_{x}f|_{\gamma,U};
				\end{align}
				
				\item  $f\in C^1(\overline{U})\cap C^2(U)$ belongs to $C^{2+\gamma}(U)$ if the function $x\partial_x^2$ can be continuously extend to $x=0$ and vanish, and the local $\mathcal{C}^{2+\gamma}$ norm is finite:
				\begin{align}
					||f||_{2+\gamma,U}=  |f|_{\gamma,U}+|\partial_{x}f|_{\gamma,U}+|x\partial_{x}^{2}f|_{\gamma,U}.
				\end{align}
				
			\end{enumerate}
		\end{enumerate}
		\begin{remark}\label{int:neutral}
			At neutral quadratic endpoint, interchanging the two integral signs, we have
			\begin{gather*}
				\int_{z_1}^{z_2}\int_{-\infty}^zf(s)s
				dsdz=\int_{-\infty}^{z_2}(z_2-s)f(s)ds,
			\end{gather*}
			so $\int_{-\infty}^{z_2}sf(s)ds<\infty$.
		\end{remark}
		We now build global norms on spaces of functions on $[0,1]$ out of the above local norms.
		\begin{definition}\label{def_holder} 
			Let $W_{2}\subset\subset(0,1)$ covering $[0,1]\setminus(W_{0}\cup W_{1})$ and $\varphi_{0},\varphi_{1},\varphi_{2}$ be a partition of unity subordinate to this cover. A function $f\in C^{2+\gamma}([0,1])$ if $(\varphi_{i}f)\circ\phi\in C^{2+\gamma}(W_i)$ for each $i$ and the global norm is 
			\begin{align*}
				||f||_{2+\gamma}=\underset{i}{\sum}||(\varphi_{i}f)\circ\phi||_{2+\gamma,W_{i}}.
			\end{align*}
			Motivated by \eqref{transition} below, we define $f\in C^{1+\gamma}([0,1])$ if $\phi'_{i}\cdot(\varphi_{i}f)\circ\phi\in C^{1+\gamma}(W_{i})$ for each $i$ and the global norm is
			\begin{align*}
				||f||_{1+\gamma}=\underset{i}{\sum}||\phi'_{i}(\varphi_{i}f)\circ\phi||_{1+\gamma,W_{i}}.
			\end{align*}
			Finally $f\in\mathcal{D}^{\gamma}([0,1])$ if $(\phi'_{i})^{2}\cdot(\varphi_{i}f)\circ\phi\in C^\gamma(W_{i})$ for each $i$ and the global norm is 
			\begin{align*}
				||f||_{\gamma}=\underset{i}{\sum}||(\phi'_{i})^{2}(\varphi_{i}f)\circ\phi||_{\gamma,W_{i}}.
			\end{align*}
		\end{definition} 
		Different choices of coverings give rise to equivalent norms. Endowed with these norms, we now show that the domains and target spaces of $M$ are all Banach spaces.
		
		\begin{proposition}
			For $0<\gamma<1$, all the spaces defined in Definition \ref{def_holder} are all Banach spaces.
		\end{proposition}
		\begin{proof}
			Take $C^{2+\gamma}([0,1])$ as an example. Let \begin{align*}i: C^{2+\gamma}([0,1])\to C^{2+\gamma}(W_{0})\times C^{2+\gamma}(W_{1})\times C^{2+\gamma}(W_{2})\end{align*}be the inclusion by mapping $f$ to $((\varphi_{0}f)\circ\phi,(\varphi_{1}f)\circ\phi,(\varphi_{2}f)\circ\phi)$. This inclusion is closed since $\varphi_{0},\varphi_{1},\varphi_{2}$ is a partition of unity. Thus we need to show that each local $(0,\gamma)$ space defined at the beginning of this subsection is Banach.
			
			The cases when $U=(0,c]$ are verified in \cite{epstein2013degenerate}.  When $U=(-\infty,c]$, we first show that $\mathcal{D}^\gamma(U)$ is a Banach space. As usual we take a Cauchy sequence $\{u_i\}_{i\geq 1}$ in $\mathcal{D}^\gamma(U)$. The standard argument for the usual H\"older space (see \cite[Remark 3.1.3]{krylov1996lectures}) applies to show that there exists a limit $u\in C^\gamma(U)$. To show that $u$ is in $\mathcal{D}^\gamma(U)$, we only need to check that $u$ can continuously extend to $-\infty$ and $\underset{z_1\leq z_2\leq c}{\sup}|\int_{z_1}^{z_2}udz|$ exists. For the first part, $\{u_{i}(-\infty)\}_{i\geq 1}$ is convergent and the limit coincides with the limit of $u$ at $-\infty$. For the second part, for any fixed $z_1\leq z_2\leq c$, $|\int_{z_1}^{z_2}u_{i}dz|$ converges to $|\int_{z_1}^{z_2}udz|$. And since $|\int_{z_1}^{z_2}u_idz|$ are uniformly bounded for $i,z_1,z_2$, $\underset{z_1\leq z_2\leq c}{\sup}|\int_{z_1}^{z_2}udz|$ exists. Therefore we proved that $u\in\mathcal{D}^\gamma(U)$, which implies that $\mathcal{D}^\gamma(U)$ is a Banach space. 
			
			For $C^{1+\gamma}(U), C^{2+\gamma}(U)$, the above proof applies to show that there exists a limit $u\in C^{1+\gamma}(U), C^{2+\gamma}(U)$ respectively, and $\partial_zf,\partial_{zz}f\in\mathcal{D}^\gamma(U)$, so that these two spaces are also Banach spaces.
		\end{proof}
		
		\begin{gather}
			L_\gamma: C^{2+\gamma}([0,1])\longrightarrow\alpha\cdot\mathcal{D}^\gamma([0,1]),
		\end{gather}
		where $\alpha(x)=x^{m(0)}(1-x)^{m(1)}$ with $m(0)=1$ if $x=0$ is Kimura and $m(0)=0$ otherwise, while $m(1)$ is defined similarly.  We first state the main theorem of this section and leave the proof to Section 3.3.
		\begin{theorem}\label{thm 3.1}
			$L_\gamma$ is a Fredholm operator and its index is 
			\begin{gather*}
				\text{ind}\ (L_\gamma)= \kappa^++\kappa^-
			\end{gather*}
			where $\kappa^+,\ \kappa^-$ are the number of positive $b_+$ and negative $b_-$, respectively,  where $b_\pm$ is defined in \eqref{b(z)}.
		\end{theorem}

		\subsection{Null Space of $\overline{L}^*$}
		Let $\overline{L}$ denote the $C^0([0,1])$-graph closure of $L$ with domain $C^{2+\gamma}([0,1])$. 
		Having obtained the index of $L_\gamma$, we are now able to find the null space of the adjoint operator $\overline{L}^*$. We first use the maximum principle below to find the kernel space of $L_\gamma$.
		
		\begin{lemma}[Maximum Principle]\label{max_prin}
			Suppose that $w\in C^{2+\gamma}([0,1])$ is a subsolution of $L$, $Lw\geq 0$ in a neighborhood, $U$ of a transverse boundary point $p$. If $w$ attains a local maximum at $p$, then $w$ is a constant on $U$. 
		\end{lemma}
		\begin{proof}
			The case when $L$ is of generalized Kimura type is studied in \cite{epstein2013degenerate}, so we only need to consider quadratic boundaries. Let $x$ denote normalized local coordinates of $0$ so that 
			\begin{gather*}
				L=x^2\partial_x^2+(b+1)x\partial_{x}, b\geq0
			\end{gather*}
			
			By subtracting $v(-\infty)$ from $v$, we may assume that $v(-\infty)=0$. Integrating $(\partial_{z}^{2}+b\partial_{z})v$ we see that $\partial_{z}v+bv\geq 0$ in a neighborhood $U$ of $-\infty$. Thus $\partial_{z}v\geq -bv\geq 0$ in $U$. As $w$ is not a constant, we can expand this neighbourhood until some $z_0$ such that $\partial_z v>0$. Then $v(z_{0})=\int_{-\infty}^{z_{0}}\partial_{z}vdz> 0$, which contradicts the fact that $-\infty$ is a local maximum.
		\end{proof}
		
		
		\begin{theorem}\label{maximum_prin}
			$\text{dim ker}\ L_\gamma=2$ if and only if $L_\gamma$ is tangent to both endpoints; otherwise $\text{dim ker}\ L_\gamma=1$. 
		\end{theorem}
		\begin{proof}
			The kernel of $L_{\gamma}$ is in the linear space of $\{1,S(x)\}$, where $S(x)$ is the scale function of the process. For 
			\begin{equation}\label{unknownL}
				L=\widetilde{a}(x)\partial_{xx}+\widetilde{b}(x)\partial_{x},
			\end{equation}
			the scale function is defined as
			\begin{gather}\label{scale_func}
				S(x)=C\int_{}^{x}\text{exp}\left[-\int_{\frac{1}{2}}^{\eta}\frac{\widetilde{b}(\xi)}{\widetilde{a}(\xi)}d\xi\right]d\eta
			\end{gather}
			and has derivatives 
			\begin{gather*}
				S'(x)=\text{exp}\left[-\int_{}^{x}\frac{\widetilde{b}(\xi)}{\widetilde{a}(\xi)}d\xi\right], \qquad S''(x)=-S'(x)\frac{\widetilde{b}(x)}{\widetilde{a}(x)}.
			\end{gather*}
			Denote $\frac{\widetilde{b}(x)}{\widetilde{a}(x)}=\frac{c(x)}{x(1-x)}$ and $c_0=c(0)$,  $c_1=c(1)$.  The integrand $\text{exp}\left[-\int_{\frac{1}{2}}^{\eta}\frac{\widetilde{b}(\xi)}{\widetilde{a}(\xi)}d\xi\right]\sim \eta^{-c_0}, (1-\eta)^{c_1}$ as $\eta$ approaches $0+,1-$ respectively, so $S$ is integrable when $c_0<1, c_1>-1$.
			
			For Kimura point $x=0$, if $0<c_0<1$, then $xS''(x)\sim x^{-c_0}$ as $x\to 0$, which is not finite and thus not in $\mathcal{D}^\gamma$. If $c_0=0$, $S''$ is smooth at $x=0$. For a quadratic point $x=0$, if $c_0<1$, turning into heat coordinates $z=lnx$, $S(z)\sim e^{(1-c_0)z}$, it is in local $C^{2+\gamma}$ space. We conclude that $S$ is in $C^{2+\gamma}([0,1])$ if and only if $L$ is tangent to both endpoints. 
		\end{proof}
		
		\begin{definition}
			We define
			\begin{gather*}
				bP_{\text{ter}}(L)=\{\text{quadratic endpoint},P\}
			\end{gather*}
			if $L$ is transverse to both endpoints, otherwise we define
			\begin{gather*}
				bP_{\text{ter}}(L)=\{\text{quadratic endpoint, Kimura tangent endpoint}\}.
			\end{gather*}
		\end{definition}
		Here we use the notation $bP_{\text{ter}}$ from \cite{epstein2013degenerate} to denote terminal boundary.
		
		We first explain that to each element of $bP_{\text{ter}}([0,1])$ there is an element of the nullspace of $\overline{L}^{*}$. For any quadratic endpoint or Kimura tangent endpoint $p$, $p$ is a sticky boundary point so $\delta(p)$ is in $\text{ker}\ \overline{L}^{*}$. For $w\in C^{2+\gamma}([0,1])$, we have 
		\begin{align*}
			\langle L_\gamma w,\delta(p)\rangle=0,
		\end{align*}
		that is to say $\delta(p)\in \text{Ker}\  L_\gamma^*$.
		This equality still holds for $w\in \text{Dom}(\overline{L})$, where $\overline{L}$ is the $\mathcal{C}^0$ graph closure of $L_\gamma$. Hence $\delta(p)\in \text{Dom}(\overline{L}^*)$, and
		$\overline{L}^*\delta(p)=0$.
		If $L$ is transverse to both endpoints, we can explicitly construct $\mu$, which is supported on the whole interval, as follows:  
		\paragraph{Construction of $\mu$}\label{consruction of mu} We first reduce $L$ to the standard form
		\begin{gather*}
			L_z=\frac{1}{2}\partial_{zz}-\nabla U(z)\partial_z=-\frac{1}{2}\partial_z^*\partial_z
		\end{gather*}
		with $\partial^*_z=\partial_z+2\nabla U(z)$ on a probability space $(X,\mu)$. It is known that $L_z$ has an invariant measure $\mu$ 
		\begin{gather*}
			\mu(dz)=\frac{e^{-2U(z)}}{Z}dz
		\end{gather*}
		where $Z$ is a normalizing constant. We now check that $\mu$ is a probability measure in the different boundary cases.
		\begin{enumerate}
			\item \textbf{Two transverse quadratic boundary.}
			For $L=x^2(1-x)^2\partial_{xx}+x(1-x)b(x)\partial_x$, we first do a coordinate change: 
			\begin{gather*}
				z=\frac{1}{\sqrt{2}}\text{ln}\frac{x}{1-x},\ x=\frac{e^{\sqrt{2}z}}{1+e^{\sqrt{2}z}},
			\end{gather*}
			so that 
			\begin{gather*}
				L_z=\frac{1}{2}\partial_{zz}-\nabla U(z)\partial_z,\  z\in (-\infty,\infty)
			\end{gather*}
			with $\nabla U=-\frac{1}{\sqrt{2}}\left(b(x)+2x-1\right)$.  Then 
			\begin{gather*}
				\mathit{e^{-2U(z)}=
					\begin{cases}
						\begin{array}{cc}
							\Theta(e^{\sqrt{2}(b(0)-1)z}),\ z\to -\infty\\
							\Theta(e^{\sqrt{2}(b(1)+1)z}),\ z\to +\infty
						\end{array}
				\end{cases}}.
			\end{gather*}
			Thus, $Z<\infty$ exactly when $b(0)>1, b(1)<-1$, i.e. both quadratic endpoints are transverse.
			\item \textbf{One transverse Kimura and one quadratic boundary.} For $L=x^2(1-x)\partial_{xx}+xb(x)\partial_x$, we do a coordinate change by letting 
			\begin{gather*}\frac{\partial z}{\partial x}=\frac{1}{\sqrt{2}x\sqrt{1-x}}
			\end{gather*}
			so that 
			\begin{gather*}
				L_z=\frac{1}{2}\partial_{zz}-\nabla U(z)\partial_z,\  z\in (-\infty,0]
			\end{gather*}
			with
			$\nabla U=-\frac{1}{\sqrt{2}}(\frac{3x-2+2b(x)}{2\sqrt{1-x}})$. Then \begin{gather*}
				\mathit{e^{-2U(z)}=
					\begin{cases}
						\begin{array}{cc}
							\Theta(e^{\sqrt{2}(b(0)-1)z}),\ z\to -\infty\\
							\Theta(z^{-1-2b(1)}),\ z\to 0
						\end{array}
				\end{cases}}.
			\end{gather*}
			Thus, $Z<\infty$ exactly when $b(0)>1, b(1)<0$, i.e. both endpoints are transverse.
			\item \textbf{Two transverse Kimura boundary.}
			For $L=x(1-x)\partial_{xx}+b(x)\partial_x$, we do a coordinate change 
			\begin{align*}\frac{\partial z}{\partial x}=\frac{1}{\sqrt{2x(1-x)}}\end{align*}
			so that 
			\begin{gather*}
				L_z=\frac{1}{2}\partial_{zz}-\nabla U(z)\partial_z,\  z\in [0,\frac{\pi}{\sqrt{2}}]
			\end{gather*}
			with $\nabla U=-\frac{x-1/2+b(x)}{\sqrt{2x(1-x)}}$. Then 
			\begin{gather}\label{inv_kq}
				\mathit{e^{-2U(z)}=
					\begin{cases}
						\begin{array}{cc}
							\Theta(z^{2b(0)-1}),\ z\to 0\\
							\Theta(z^{-1-2b(1)}),\ z\to\frac{\pi}{\sqrt{2}}
						\end{array}
				\end{cases}}.
			\end{gather}
			Again, $Z<\infty$ exactly when $b(0)>0, b(1)<0$, i.e. both endpoints are transverse.
		\end{enumerate}
		
		\begin{proposition}
			\begin{gather*}
				\text{dim ker}\ \overline{L}^*=|bP_{\text{ter}}(L)|. 
			\end{gather*}
		\end{proposition}
		\begin{proof}
			We know that $\text{dim\ ker}\ L_\gamma=2$ if $L$ is tangent to both endpoints and otherwise $\text{dim\ ker}\  L_\gamma=1$. This is equivalent to:
			\begin{align*}
				\text{dim\ ker}\ L_{\gamma}=\text{max}\{1,|\text{tangent\ points}|\}.
			\end{align*}
			By Theorem \ref{thm 3.1}, $\text{ind}(L_{\gamma})=\kappa^{+}+\kappa^{-}=|\text{tangent\ quadratic endpoints}|$, so $\text{dim ker}\ L_\gamma^*=|b P_{ter}(L_\gamma)|$, where
			\begin{gather*}
				bP_{ter}(L_\gamma)=\begin{cases}
					\begin{array}{cc}
						P & \text{\ensuremath{L} is transverse/neutral to both endpoints,}\\
						\text{Kimura tangent endpoint} & \text{otherwise}.
				\end{array}\end{cases}
			\end{gather*}
			
			Write 
			\begin{gather*}
				C^0([0,1])^*= C^0_0([0,1])^*\oplus A^*
			\end{gather*}
			where $A^*=\{\delta(p)|\ p \ \text{quadratic}\}$.   $A^*\subset\text{ker}\ \overline{L}^*$. Since $\widetilde{a}\cdot\mathcal{D}^{\gamma}([0,1])$ is dense in the subspace $C_0^0([0,1])$, every $\mu\in\text{ker}\ {L_\gamma}^*$ can be uniquely extended to a measure in $\text{ker}\ \overline{L}^*\cap C^0_0([0,1])^*$, so there is an inclusion map
			\begin{gather*}
				i:\text{ker}\ \overline{L}^{*}\cap C^0_0([0,1])^{*}\longrightarrow \text{ker}\ L^*_\gamma.
			\end{gather*}
			
			And $\text{codim}(i)=1$ if there are one or more neutral point(s) and no tangent point, while $\text{codim}(i)=0$. This is equivalent to
			\begin{gather*}
				\text{codim}(i)=|bP_{\text{ter}}( L_\gamma)|-|bP_{ter}(L)|+|\text{quadratic endpoint}|.
			\end{gather*}
			
			In conclusion
			\begin{gather*}
				\text{dim ker}\ \overline{L}^{*}=\text{dim ker}\ L_{\gamma}^*-\text{codim}(i)+|\text{quadratic\ point}|=|bP_{ter}(L)|.
			\end{gather*}
		\end{proof}
		The following table summarizes the invariant measures found in the ten different cases of interest:
		\begin{table}[h]
			\centering 
			\begin{threeparttable}
				\begin{tabular}{ccccc}
					& K Trans & K Tan & Q Trans & Q Tan  \tnote{**}\\
					\midrule\midrule
					Kimura Transverse (K Trans) & $\mu$\tnote{*} &  $\delta_1$ &  $\mu$, $\delta_1$ & $\delta_1$ \\ 
					Kimura Tangent (K Tan) & $\delta_0$ & $\delta_0$, $\delta_1$ & $\delta_0$ & $\delta_0$, $\delta_1$ \\
					Quadratic Transverse (Q Trans) & $\mu$, $\delta_0$ & $\delta_0$, $\delta_1$ & $\mu$, $\delta_0$, $\delta_1$ & $\delta_0$, $\delta_1$\\
					Quadratic Tangent (Q Tan)\tnote{**} & $\delta_0$ &$\delta_0$, $\delta_1$ & $\delta_1$ &$\delta_0$, $\delta_1$\\
					\midrule\midrule
				\end{tabular}
				\begin{tablenotes}
					\item[*] $\mu$ refers to an invariant measure supported on (0,1)
					\item[**] include neutral case
				\end{tablenotes} 
			\end{threeparttable}
			\caption{Invariant measures for all cases of boundary types at $x=0$ (rows) and at $x=1$ (columns).}
			\label{tab:1d}
		\end{table}
		
		\subsection{Proof of Theorem \ref{thm 3.1}}
		For convenience of computation, we fix the spatial domain as the interval $[0,3]$ instead of $[0,1]$ whenever convenient in this subsection.
		\subsubsection{Outline of proof}
		It remains to characterize the range of $L_\gamma$. First, we turn the second-order operator $L_\gamma$ to a first-order system $M$ in \eqref{M}, with an index equal to the index of $L_\gamma$ (see Lemma \ref{lmm:3.2} below). Next we continuously deform $M$ to another first-order system $\widetilde{M}$ with constant coefficients in the vicinity of the two endpoints. Such a deformation does not change the index (see Proposition \ref{prop:index} below), so the original problem now is equivalent to the easier linear system $\widetilde{M}$. These constructions are presented in subsection 3.3.2.
		
		The problem $\widetilde{M}u=f=\left(\begin{array}{c}
			f_1\\
			f_2
		\end{array}\right)$, given $u(1)\in\mathbb{R}^2$, has a uniquely defined solution $u$. In order for $u$ to belong to the domain space, we observe that $\alpha f_1$ has to vanish at tangent Kimura endpoints and $u(1), f$ have to be related by
		\begin{gather}\label{relation}
			E_{-}u(1)=I_1f, \quad E_{+}\Lambda u(1)=-I_2f;
		\end{gather}
		see the definition of $I_1, I_2, E_\pm$, and $\Lambda$ in Section 3.4. Conversely if  $\alpha f_1$ vanishes at tangent Kimura endpoints and if we can find $u(1)$ satisfying  the relation \eqref{relation}, then the uniquely defined $u$ is the solution of $\widetilde{M}u=f$ (see Lemma \ref{check1}, Lemma \ref{check2}). Thus $f$ in the range of $\widetilde{M}$ is equivalent to the existence of $u(1)\in\mathbb{R}^2$ satisfying \eqref{relation}. This construction is presented in subsection 3.3.3.
		
		We next show in Lemma \ref{lm 3.4} that the above constraint is equivalent to
		\begin{gather*}
			\Lambda I_1f+I_2f\in \text{ker}(E_{+})+\Lambda\cdot \text{ker}(E_{-}).
		\end{gather*}
		We thus define the linear map
		\begin{gather*}
			\Phi: C^{1+\gamma}([0,3])\times {\mathcal{D}^\gamma([0,3])}\longrightarrow \frac{\mathbb{R}^2}{\text{ker}(E_{+})+\Lambda \text{ker}(E_{-})},
		\end{gather*}
		by assigning $f$ to the coset $[\Lambda I_1f+I_2f]$. Clearly the range of $\widetilde{M}$ lies in the kernel of $\Phi$. A computation in Lemma \ref{phi} shows that $\Phi$ is surjective and $\text{dim\ ker}(\Phi)/R(\widetilde{M})=|\text{kimura tangent points}|$. We now have all the ingredients to compute the codimension of $\widetilde{M}$. The final results and proofs are given in subsection 3.3.4.

		\subsubsection{Reduction to a first order system}
		To simplify notation, we consider $L_\gamma$ defined on the interval $[0,3]$. We rewrite the system $L_\gamma u=f$ as a first-order system $M$:
		\begin{gather}\label{M}
			M: C^{2+\gamma}([0,3])\times C^{1+\gamma}([0,3])\longrightarrow C^{1+\gamma}([0,3])\times\mathcal{D}^{\gamma}([0,3]).
		\end{gather}
		associated to the expression 
		\begin{gather*}
			\begin{pmatrix}u_0\\u_1
			\end{pmatrix}\mapsto   \begin{pmatrix}u_0'\\u_1'
			\end{pmatrix}+A(L_{\gamma})  \begin{pmatrix}u_0\\u_1
			\end{pmatrix}=\begin{pmatrix}
				u_0'-u_1\\u_1'+\frac{b}{a}u_1\end{pmatrix} ,\qquad A(L_\gamma):=\left(\begin{array}{cc}
				0 & -1\\
				0 & \frac{b}{a}
			\end{array}\right).
		\end{gather*}
		where $a, b$ are the coefficients of second-order and first-order term of $L$ under heat coordinate $\phi$, i.e. $L_\phi=a(x)\partial_{xx}+b(x)\partial_x$.
		
		For a smooth coordinate change $z=z(x)$, we let $L_{z}$ be the operator corresponding to $L$ under coordinate change and define $A(L_{z})$ and $M_{z}$ as above.
		If $M\begin{pmatrix}u_0\\u_1\end{pmatrix}=\begin{pmatrix}f_0\\f_1\end{pmatrix}$, then
		\begin{gather}\label{transition}
			M_{z}\begin{pmatrix}u_0[x(z)]\\u_1[x(z)]x'(z)\end{pmatrix}=\begin{pmatrix}f_0[x(z)]x'(z)\\f_1[x(z)](x'(z))^{2}\end{pmatrix}.
		\end{gather}
		
		In the local charts $([0,1],\phi^{-1}),([2,3],\phi^{-1})$,
		$A(t)=\left(\begin{array}{cc}
			0 & -1\\
			0 & b(z)
		\end{array}\right)$ when 0 is a quadratic point with heat coordinate, while $A(t)=\left(\begin{array}{cc}
			0 & -1\\
			0 & \frac{{b}(x)}{x}
		\end{array}\right)$ when $0$ is a Kimura point.

		\begin{lemma}\label{lmm:3.2}
			Assume that $M$ is Fredholm, then $L_\gamma$ is also Fredholm and
			\begin{gather}\label{idx}
				\text{ind}(L_\gamma)=\text{ind}(M). 
			\end{gather}
		\end{lemma}
		\begin{proof}
			There is an isomorphism between $\text{ker}(M)\longrightarrow \text{ker}(L_{\gamma})$ by mapping $(u,u')$ to $u$, so the dimensions of the kernel spaces are the same. 
			
			If $(f_{0},f_{1})\in R(M)$ and $M\begin{pmatrix}u_0\\u_1\end{pmatrix}=\begin{pmatrix}f_0\\f_1\end{pmatrix}$, then $L_\gamma u_0=\widetilde{a} f_{0}'+\widetilde{b} f_{0}+\widetilde{a} f_{1}$. Based on this observation we define the map
			\begin{gather*}
				\pi: C^{1+\gamma}([0,3])\times\mathcal{D}^{\gamma}([0,3])\longrightarrow \alpha\cdot\mathcal{D}^{\gamma}([0,3])\\
				(f_{0},f_{1})\mapsto \widetilde{a} f_{0}'+\widetilde{b} f_{0}+\widetilde{a} f_{1}.
			\end{gather*}
			We claim that 
			\begin{align}\label{codim}
				\text{codim ran}(\pi)=\text{codim\ ran}(L_{\gamma})-\text{codim ran}(M). 
			\end{align}
			We first show that 
			\begin{gather}\label{14}
				f\in \text{ran}(M)\Leftrightarrow\pi f\in \text{ran}(L_{\gamma}). 
			\end{gather}
			The inclusion part is straightforward. For the converse part, if $\pi(f_{0},f_{1})=L_{\gamma}u\in \text{ran}(L_{\gamma})$, we verify that $M\begin{pmatrix}u\\u'-f_0\end{pmatrix}=\begin{pmatrix}f_0\\f_1\end{pmatrix}\in \text{ran}(M)$. Next we show that 
			\begin{gather}\label{15}
				\pi(\text{ran}(M))=\text{ran}(L_{\gamma}).
			\end{gather} Again the inclusion part is straightforward. For the converse part, if $g=L_{\gamma}u\in \text{ran}(L_\gamma)$, then $\pi(0,\frac{g}{\widetilde{a}})=g\in \text{ran}(L_\gamma)$. Note that by \eqref{14}, this implies that $(0,\frac{g}{\widetilde{a}})\in \text{ran}(M)$. \eqref{14} and \eqref{15} imply that the quotient map defined by
			\begin{gather*}
				\pi': [C^{1+\gamma}([0,3])\times\mathcal{D}^{\gamma}([0,3])]/\text{ran}(M)\longrightarrow \alpha\cdot\mathcal{D}^{\gamma}([0,3])/\text{ran}(L_\gamma)
			\end{gather*}
			is injective. So $\text{codim}\  \text{ran}(\pi')=\text{codim}\ \text{ran}(L_{\gamma})-\text{codim}\ \text{ran}(M)$. At the same time,  $\pi(\text{ran}(M))=\text{ran}(L_{\gamma})$ indicates $\text{codim}\ \text{ran}(\pi')=\text{codim} \ \text{ran}(\pi)$ so that \eqref{codim} holds. Clearly, $\pi$ is surjective since $\pi (0,\frac{h}{\widetilde{a}})=h$. Combined with \eqref{codim}, this proves \eqref{idx}.
		\end{proof}
		\paragraph{Index of the model operator $\widetilde{M}$}
		Alongside $M$, we consider the operator $\widetilde{M}$ associated with a continuous function $A(t)$ such that
		\begin{gather*}
			A(t)=A_{-}\ \text{for}\ t\leq 1, A(t)=A_{+}\ \text{for}\ t\geq 2
		\end{gather*}
		where  $A_{-}=\left(\begin{array}{cc}
			0 & -1\\
			0 & b_{-}
		\end{array}\right)$ if 0 is a quadratic point, $A_{-}=\left(\begin{array}{cc}
			0 & -1\\
			0 & \frac{b_{-}}{x}
		\end{array}\right)$ if 0 is a kimura point, and $A_+$ is similarly defined. Since $M$ and $\widetilde{M}$ can be reduced to each other by a continuous deformation in the class of Fredholm operators, we have
		\begin{proposition}\label{prop:index}
			The operator $M$ is Fredholm if and only if the operator $\widetilde{M}$ is Fredholm and 
			\begin{gather}\label{prop:fred_euiv}
				\text{ind}(M)=\text{ind}(\widetilde{M}).
			\end{gather}
		\end{proposition}
		\begin{proof}
			We need to show that $M-\widetilde{M}$ is a compact map since then by \cite[§27.1 Theorem 3]{lax2014functional}, $\widetilde{M}$ has the same index as $M$. 
			
			Given a bounded sequence $u_k$ in $C^{2+\gamma}([0,3])\times C^{1+\gamma}([0,3])$, we need to show that there exists a convergent subsequence $v_k$ of $(M-\widetilde{M})u_k$. $v_k$ convergent in $C^{1+\gamma}([0,3])\times\mathcal{D}^\gamma([0,3])$ is equivalent to  $v_{k}|_{[0,1]},v_{k}|_{[1,2]},v_{k}|_{[2,3]}$ convergent in the local spaces, respectively. Thus it suffices to show that $(M-\widetilde{M})|_{[0,1]},(M-\widetilde{M})|_{[2,3]}$ are compact.
			
			To check both Kimura and quadratic cases, we assume that $L_0$ is Kimura type and $L_1$ is of quadratic type. 
			
			At $x=0$, given a bounded sequence $(u_n,q_n)$ in $C^{2+\gamma}([0,1])\times C^{1+\gamma}([0,1])$ , we have \begin{gather*} (M-\widetilde{M})(u_n,q_n)=(0,\frac{b(x)-b_{-}}{x}q_{n}).
			\end{gather*} 
			Since $\frac{b(x)-b_{-}}{x}$ is smooth and bounded, it remains to show by definition of $\mathcal{D}^\gamma([0,1])$ that $xq_n$ has a convergent subsequence in $C^\gamma([0,1])$.
			
			Since $\{q_n,xq_n'\}$ are bounded in $C^\gamma([0,1])$, then $\{xq_n\}$ are uniformly bounded and uniformly equicontinuous. Therefore by the Arzela-Ascoli theorem, there exists a convergent subsequence $v_n$ of $xq_n$ in $C^0([0,1])$. For the $[\cdot]_\gamma$ part, notice that for $g_n(y)=v_n(x)$ under the coordinate change $y=x^\gamma$, 
			\begin{gather}\label{vn}
				[v_n]_\gamma=O(||g_n'||_\infty).
			\end{gather}
			Since $g_n'=\frac{1}{\gamma}v_n'x^{1-\gamma}$ and $\{v'_n\}$ are uniformly bounded in $C^\gamma([0,1])$, $\{g'_n\}$ are uniformly bounded and uniformly equicontinuous. Thus, there exists a subsequence of $g'_n$ convergent in $C^0([0,1])$. By \eqref{vn}, the corresponding subsequence of $v_n$ converges in $C^\gamma([0,1])$. Therefore $(M-M')|_{[0,1]}$ is a compact operator. 
			
			At $x=3$, we prove the result using heat coordinates in $[-\infty,0]$. Given a bounded sequence $(u_{n},q_{n})$ in $C^{2+\gamma}((-\infty,0])\times C^{1+\gamma}((-\infty,0])$, write \begin{gather*} (M-\widetilde{M})(u_n,q_n)=(0,(b(z)-b_{+})q_{n}).
			\end{gather*} 
			Since $\{q'_n\}$ is bounded in $C^\gamma(T)$for any compact set $T$, there exists a convergent subsequence of $\{q_n\}$ in $C^\gamma(T)$. We choose a convergent subsequence $\{q^k_1\}_{k\geq 1}$ on $[-1,0]$, and $\{q^k_2\}_{k\geq 1}$ a convergent subsequence of $\{q^k_1\}$ on $[-2,0]$. Iteratively, we choose $\{q^k_n\}_{k\geq 1}$ to be a convergent subsequence of $\{q^k_{n-1}\}$ on $[-n,0]$.
			
			We show below that the diagonal sequence $\{\left(b(z)-b_{+}\right)q^k_k\}_{k\geq 1}$ converges in $\mathcal{D}^\gamma((-\infty,0])$. Since $||q_k^k||_\gamma$ is bounded and $b(z)-b_{-}$ tends to 0 as $z$ tends to $-\infty$, there exist $N> 0, K>0$ such that
			\begin{gather*}
				\max\Big\{|(b(z)-b_{-})q_k^k|,\quad [(b(z)-b_{-})q_k^k]_\gamma, \quad \underset{z_1\leq z_2\leq -N}{\sup}\left|\int_{z_1}^{z_2}(b(z)-b_{-})q_k^kdz\right|\Big\} \leq\epsilon/4, \text{ on }\ (-\infty, -N]\\
				\{q^k_k\}_{k\geq K}\ \text{converge in}\ C^\gamma([-N,0]).
			\end{gather*} 
			So we can pick $N'$ such that  $||(b(z)-b_{-})q^{m}_{m}-(b(z)-b_{-})q^{n}_{n}||_{\gamma, [-N,0]}<\frac{\epsilon}{2}$ when $m,n>N'$. Combining the above, we have $||(b(z)-b_{-})q^{m}_{m}-(b(z)-b_{-})q^{n}_{n}||_{\gamma, (-\infty,0]}<\epsilon$ when $m,n>N'$. We proved that the diagonal sequence $\{\left(b(z)-b_{+}\right)q^k_k\}_{k\geq 1}$ converged in $\mathcal{D}^\gamma((-\infty,0])$. Therefore $(M-M')|_{[2,3]}$ is a compact operator. It is not hard to show that $(M-M')|_{[1,2]}$ is also a compact operator so that $M-\widetilde{M}$ is compact.
		\end{proof}
		
		\subsubsection{Representation of solutions of the modeling operator}
		The results in the previous subsection show that $\text{ind} (L_\gamma)=\text{ind}(\widetilde{M})$ if $\widetilde{M}$ is Fredholm. Now we turn to proving that $\widetilde{M}$ is Fredholm and computing the index of $\widetilde{M}$. Recall that
		\begin{gather}
			\widetilde{M}: C^{2+\gamma}([0,3])\times C^{1+\gamma}([0,3])\longrightarrow C^{1+\gamma}([0,3])\times\mathcal{D}^{\gamma}([0,3])
		\end{gather}
		is associated to the expression 
		\begin{gather*}
			\begin{pmatrix}u_0\\u_1
			\end{pmatrix}\mapsto \begin{pmatrix}u_0'\\u_1'
			\end{pmatrix}+A(t)\begin{pmatrix}u_0\\u_1
			\end{pmatrix}, \qquad A(t)=A_{-}\ \text{for}\ t\leq 1, \ \ A(t)=A_{+}\ \text{for}\ t\geq 2.
		\end{gather*}

		Let $f\in C^{1+\gamma}([0,3])\times\mathcal{D}^\gamma([0,3])$ be given, and suppose $u\in C^{2+\gamma}([0,3])\times C^{1+\gamma}([0,3])$ solves
		\begin{gather*}
			\frac{du}{dt}+Au=f.
		\end{gather*}
		Let $Y=Y(t)$ denote the fundamental matrix associated to $\widetilde{M}$, defined in different intervals by 
		\begin{gather*}
			\frac{dY}{dt}+AY(t)=0,\ Y(1)=I, \text{on}\ [0,2)\\
			\frac{dY}{dt}+AY(t)=0,\ Y(2+)=I, \text{on}\ (2,3].
		\end{gather*}
		Then $\Lambda:=Y(2-)$ is an invertible matrix and
		\begin{gather*}
			u(2)=\Lambda u(1)+\Lambda\int_{1}^{2}Y(s)^{-1}f(s)ds.
		\end{gather*}
		Given $u(1)$, we can uniquely determine $u(t)$ on $[0,3]$:
		\begin{gather}
			\label{solution_u1}
			u(t)=Y(t)u(1)-\int_t^1 Y(t)Y^{-1}(s)f(s)ds,\ t\leq 1\\
			\label{solution_u2}
			u(t)=Y(t)\left(\Lambda u(1)+\Lambda\int_1^2Y(s)^{-1}f(s)ds\right)+\int_2^tY(t)Y^{-1}(s)f(s)ds,\ t\geq 2.
		\end{gather}
		We analyze the above solutions for different boundary points. For quadratic endpoints, we directly assume $t$ is in heat coordinates.
		
		\paragraph{1. Quadratic endpoints with $b_\pm\neq 0$,}
		$A_\pm=\left(\begin{array}{cc}
			0 & -1\\
			0 & b_\pm 
		\end{array}\right)$ constant matrix.  Then 
		\begin{align}
			Y(t)=\left\{ \begin{array}{cc}
				e^{-(t-1)A_{-}}, & t\in (-\infty,1] \\
				e^{-(t-2)A_{+}}, & t\in [2,+\infty).
			\end{array}\right.
		\end{align}  $A_\pm$ has two different eigenvalues $0, b_\pm$.
		We let $E_-$ be the spectral projector associated to the set of all positive eigenvalues of $A_-$, and $E_+$ be the spectral projector associated to the set of all negative eigenvalues of $A_+$. Apply the operator $E_{-}, E_{+}$ on $\eqref{solution_u1}, \eqref{solution_u2}$, and let $t$ go to $-\infty,\infty$ to get:
		\begin{gather}
			\label{e1}
			E_{-}u(1)=\int_{-\infty}^{1}e^{{(s-1)}A_{-}}E_{-}f(s)ds,\\
			\label{e2}
			E_{+}\Lambda u(1)=-E_{+}\Lambda\int_{1}^{2}Y(s)^{-1}f(s)ds -\int_{2}^{\infty}E_+e^{(s-2)A_{+}}f(s)ds.
		\end{gather}
		Using this relation, we recast $u(t)$ for $t<1$ as:
		\begin{gather}\label{u1}
			u(t)=e^{-(t-1)A_{-}}(I-E_{-})u(1)-\int_{t}^{1}e^{-(t-s)A_{-}}(I-E_{-})f(s)ds+\int_{-\infty}^{t}e^{-(t-s)A_{-}}E_{-}f(s)ds
		\end{gather}
		and for $t>2$ as:
		\begin{gather}\label{u2}
			u(t)= e^{-(t-1)A_{+}}(I-E_{+})\Lambda\left[u(1)+\int_1^2Y(s)^{-1}f(s)ds\right]\\ +\int_2^{t}e^{-(t-s)A_{+}}(I-E_{+})f(s)ds
			-\int_{t}^{\infty}e^{-(t-s)A_{+}}E_{+}f(s)ds. \nonumber
		\end{gather}
		
		\paragraph{2. Kimura transverse point} In this case
		\begin{gather*}
			Y(t)=\left(\begin{array}{cc}
				1 & F(t^{-b_{-}})-F(1)\\
				0 & t^{-b_{-}}
			\end{array}\right), \ t\in [0,1], \\
			Y(t)=\left(\begin{array}{cc}
				1 & F((\frac{t}{2})^{-b_+})-F(1)\\
				0 & (\frac{t}{2})^{-b_+}.
			\end{array}\right), \ t\in [2,3],
		\end{gather*}
		where $F$ denotes the antiderivative.
		We let $E_\pm=\left(\begin{array}{cc}
			0 & 0\\
			0 & 1
		\end{array}\right)$.
		Multiplying by $E_{-}Y(t)^{-1}$ both sides of \eqref{solution_u1}, \eqref{solution_u2}, and letting $t\to 0, t\to 3$, respectively, we obtain 
		\begin{gather}
			\label{relation_1}
			E_{-}u(1)=\int_{0}^{1}E_{-}Y(s)^{-1}f(s)ds,\\
			\label{relation_2}
			E_+\Lambda u(1)=-E_+\Lambda\int_{1}^{2}Y(s)^{-1}f(s)d-\int_{2}^{3}E_+Y^{-1}(s)f(s)ds.
		\end{gather}
		Using this relation we recast $u(t)$ for $t<1$ as:
		\begin{gather}\label{u7}
			u(t)=Y(t)(I-E_-)u(1)+ Y(t)\int_0^tE_{-}Y^{-1}(s)f(s)ds-Y(t)\int_t^1(I-E_-)Y^{-1}(s)f(s)ds
		\end{gather}
		and for $t>2$ as:
		\begin{gather}\label{u8}
			u(t)=Y(t)(I-E_{+})\Lambda\left(u(1)+\int_{1}^{2}Y(s)^{-1}f(s)ds\right)\\
			+\int_2^tY(t)(I-E_+)Y^{-1}(s)f(s)ds- \int_t^3Y(t)E_+Y^{-1}(s)f(s)ds. \nonumber
		\end{gather}
		
		\paragraph{3. Quadratic endpoints with $b_\pm=0$} In this case
		\begin{gather*}
			Y(t)=\left(\begin{array}{cc}
				1 & t-1\\
				0 & 1
			\end{array}\right), \ t\in [0,1],\qquad \quad
			Y(t)=\left(\begin{array}{cc}
				1 & t-2\\
				0 & 1.
			\end{array}\right), \ t\in [2,3].
		\end{gather*}
		We let $E_\pm=\left(\begin{array}{cc}
			0 & 0\\
			0 & 1
		\end{array}\right)$.
		Multiply by $e^{t}$ on both sides of \eqref{solution_u1}, \eqref{solution_u2} and let $t\to\pm\infty$. We again obtain
		\eqref{relation_1}, \eqref{relation_2} with the integral lower/upper points $0,3$ replaced by $-\infty, \infty$. So we can re-express $u(t)$ as \eqref{u7}, \eqref{u8} again with integral lower/upper points $0,3$ replaced with $-\infty, \infty$.

		\paragraph{4. Kimura tangent point} In this case, $A_\pm=\left(\begin{array}{cc}
			0 & -1\\
			0 & 0
		\end{array}\right)$. We observe that if $f=(f_0,f_1)=\widetilde{M}(u_0, u_1)\in \text{ran}(\widetilde{M})$, then $f_1=u_1'$, so that $s(s-3)f_1(s)$ vanishes at Kimura tangent endpoints by definition of $C^{1+\gamma}([0,3])$. We next let $E_\pm=\left(\begin{array}{cc}
			0 & 0\\
			0 & 0
		\end{array}\right)$. Obviously \eqref{relation_1}, \eqref{relation_2} are true then. So we can also recast $u(t)$ as \eqref{u7}, \eqref{u8}.
		
		\medskip
		Summarizing these boundary cases, we make the following definition.
		\begin{definition}
			We define
			\begin{gather*}
				I_1f=\int_{l}^1E_{-}Y(s)^{-1}f(s)ds, \\
				I_2f=E_{+}\Lambda{\int_{1}^{2}}Y(s)^{-1}f(s)ds+ \int_2^rE_+Y^{-1}(s)f(s)ds,
			\end{gather*}
			with $l,r=-\infty,\infty$ at a quadratic endpoint and $l,r=0,3$ at a Kimura endpoint.
			\end{definition}
			In all cases, when $f$ satisfying the corresponding boundary conditions is in the range of $\widetilde{M}$, then $I_1f, I_2f$ are related by
			\begin{gather}\label{11}
				E_{-}u(1)=I_1f, \quad E_{+}\Lambda u(1)=-I_2f.
			\end{gather}

				For $f$ satisfying the relation \eqref{11}, we fix a constant $u(1)$ as in \eqref{11}, and then define $u(t)$ on the whole domain by \eqref{u7} and \eqref{u8} (or \eqref{u1} and \eqref{u2} for quadratic endpoints). We now show that such a $u$ is indeed the solution. 

				\begin{lemma}\label{check1}
					For $U=(-\infty,c]$ or $U=[c,\infty)$, if $f=(f_0,f_1)\in C^{1+\gamma}(U)\times\mathcal{D}^{\gamma}(U)$ together with some $u(1)\in\mathbb{R}^2$ satisfy \eqref{11}, then $u(t)$ defined by \eqref{u1},\eqref{u2} is in $C^{2+\gamma}(U)\times C^{1+\gamma}(U)$ and solves 
					\begin{align*}
						\frac{du}{dt}+Au=f.
					\end{align*}
				\end{lemma}
				\begin{proof}
					When $b_{\pm}\neq 0$, $A_{\pm}$ can be diagonalized as 
					\begin{align*}
						A_{\pm}=\left(\begin{array}{cc}
							1 & -1\\
							0 & b_{\pm}
						\end{array}\right)\left(\begin{array}{cc}
							0 & 0\\
							0 & b_{\pm}
						\end{array}\right)\left(\begin{array}{cc}
							1 & -1\\
							0 & b_{\pm}
						\end{array}\right)^{-1}=CB_\pm C^{-1}.
					\end{align*}
					By changing a basis $u\mapsto C^{-1}u$, we may assume that $A_{\pm}$ has the diagonal form $B_\pm$ and $E_{\pm}$ is the corresponding projector of $A_{\pm}$. For convenience we only prove \eqref{u1} as \eqref{u2} is obtained similarly. 
					
					For the solution $u=(u_{0},u_{1})$, it is straightforward to see that $u_{0}\in C^2((-\infty,c])$, $\underset{z\to -\infty}{\lim}\partial_{z}u=\underset{z\to\ -\infty}{\lim}\partial_{z}^{2}u=0$ and satisfies $\partial_{z}^{2}u+b_{-}\partial_{z}u=f_{1}'+f_{2}+bf_{1}$. Hence the classical result of elliptic operator in H\"older spaces gives that $u_{0}\in C^{2+\gamma}((-\infty,c])$. Then $u_{1}=u_{0}'-f_{1}\in C^{1+\gamma}((-\infty,c])$.
					
					When $b_\pm=0$, we explicitly express $u(t)$ for $t<1$ as follows:
					\begin{gather}\label{u_0,1}
						u_0(t)=u_0(1)-u_1(1)+tu_1(t)-\int_t^1[f_0(s)-sf_1(s)]ds,\\
						u_1(t)=u_1(1)-\int_t^1f_1(s)ds.
					\end{gather}
					Since $f$ satisfies \eqref{e1},
					by Remark \ref{int:neutral}, $u_0, u_1$ are integrable up to $-\infty$. $u_1'(t)=f_1(t)\in C^\gamma((-\infty, c])$ so $u_1\in C^{1+\gamma}((-\infty, c])$. And  \[u_0'(t)=u_1(1)+f_0(t)-\int^0_tf_1(s)ds,\ \  u_0''(t)=f_0'(t)+f_1(t).\] 
					$u_0''\in C^\gamma((-\infty,c])$ and $u_0'$ is integrable up to $-\infty$, so $u_0'\in C^{1+\gamma}((-\infty,c])$. Thus, $u\in C^{2+\gamma}((-\infty, c])\times C^{1+\gamma}((-\infty, c])$ and solves $\frac{du}{dt}+Au=f$. 
				\end{proof}
				
				\begin{lemma}\label{check2}
					For $U=[0,1]$ or $U=[2,3]$, if $f=(f_0,f_1)\in C^{1+\gamma}(U)\times\mathcal{D}^{\gamma}(U)$ together with some $u(1)\in\mathbb{R}^2$satisfies \eqref{11} and $s(s-3)f_1(s)$ vanishes at Kimura tangent endpoints, then $u(t)$ defined by \eqref{u7},\eqref{u8} is in $C^{2+\gamma}(U)\times C^{1+\gamma}(U)$ and solves
					\begin{gather*}
						\frac{du}{dt}+Au=f.
					\end{gather*}
				\end{lemma}
				\begin{proof}
					At a transverse Kimura point, we write $u(t)$ for $t<1$ as
					\begin{gather}
						u(t)=\left(\begin{array}{cc}
							1 & 0\\
							0 & 0
						\end{array}\right)u(1)+ \int_0^t\left(\begin{array}{cc}
							0 & s^{b_-}G(t)\\
							0 & t^{-{b_-}}s^{b_-}
						\end{array}\right)f(s)ds-\int_t^1\left(\begin{array}{cc}
							1 & s^{b_-}G(s)\\
							0 & 0
						\end{array}\right)f(s)ds,
					\end{gather}
					where $G(t)=F(t^{-{b_-}})-F(1)$.
					By direct computation we have
					\begin{gather*}
						u_1'(t)=f_1(t)-at^{-({b_-}+1)}\int_0^ts^{b_-}f_1(s)ds,\\
						u_0'(t)=t^{-{b_-}}\int_0^ts^{{b_-}}f_1(s)ds+f_0(t), u_0''(t)=f_1(t)-{b_-}t^{-({b_-}+1)}\int_0^ts^{b_-}f_1(s)ds+f_0'(t).
					\end{gather*}
					Thus $u_0'', u_1'\in\mathcal{D}^\gamma([0,1])$ and by definition
					$u\in C^{2+\gamma}(U)\times C^{1+\gamma}(U)$ and solves $\frac{du}{dt}+Au=f$.

					At a tangent Kimura point, for $t\in [0,1]$, we write $u(t)$ for $t<1$ as:
					\begin{gather}
						u_0(t)=u_0(1)+(t-1)u_1(1)-\int_t^1[f_0(s)+(t-s)f_1(s)]ds,\\
						u_1(t)=u_1(1)-\int_t^1f_1(s)ds.
					\end{gather}
					$u_1$ is integrable up to $t=0$ iff $\underset{s\to 0}{\lim}\ sf_1(s)=0$.
					In such a case, $u_1'(t)=f_1(t)\in\mathcal{D}^\gamma([0,1])$ and so $u_1\in C^{1+\gamma}([0,1])$. Since
					\begin{gather*}
						u_0'(t)=u_1(1)+f_0(t)-\int^1_tf_1(s)ds,\ u_0''(t)=f_0'(t)+f_1(t),
					\end{gather*}
					$u\in C^{2+\gamma}(U)\times C^{1+\gamma}(U)$ and solves $\frac{du}{dt}+Au=f$. For $t\in [2,3]$, the proof is essentially the same.
				\end{proof}

				\subsubsection{Proof of Theorem \ref{thm 3.1}}
				\begin{lemma}
					$dim\ \text{ker}(\widetilde{M})=dim\ [\text{ker}(E_-)\cap\Lambda^{-1}\text{ker}(E_+)]$.
				\end{lemma}
				\begin{proof}
					If $\widetilde{M}u=0$, then $\Lambda u(1)\in \text{ker}(E_+)$ and similarly $u(1)\in \text{ker}(E_-)$. On the other hand, we set $u(1)\in \text{ker}(E_{-})\cap\Lambda^{-1}\text{ker}(E_+)$, and define $u$ at the quadratic boundary point by
					\begin{gather*}
						\text{for}\ t\leq 1: u(t)=e^{-(t-1)A_{-}}u(1),\quad  \text{for}\ t\geq 2: u(t)=e^{-(t-2)A_{+}}\Lambda u(1)
					\end{gather*}
					and at the Kimura endpoint by 
					\begin{align*}
						\text{for}\ t\leq 1: u(t)=\Phi(t)u(1),\quad  \text{for}\ t\geq 2: u(t)=\Phi(t)\Lambda u(1)
					\end{align*}
					and $u(t)$ is uniquely characterized on $[1,2]$ by $u(1)$. Then, $u\in C^{2+\gamma}([0,3]\times C^{1+\gamma}([0,3])$ and solves $\widetilde{M}u=0$. When such $u$ exists, it is uniquely determined by $u(1)$. We thus proved that the map 
					\begin{align*}
						N: \text{ker}(\widetilde{M})\longrightarrow \text{ker}(E_{-})\cap\Lambda^{-1}\text{ker}(E_{+})
					\end{align*}
					by assigning $u$ to $u(1)$ is a bijection. Therefore $\text{dim}\ \text{ker}(\widetilde{M})=\text{dim}\ [\text{ker}(E_{-})\cap\Lambda^{-1}\text{ker}(E_{+})]$.
				\end{proof}
				\begin{lemma}\label{lm 3.4}
					The condition \eqref{11} is equivalent to 
					\begin{gather}\label{12}
						\Lambda I_1f+I_2f\in \text{ker}(E_{+})+\Lambda\cdot \text{ker}(E_{-}).
					\end{gather}
				\end{lemma}
				\begin{proof}
					Suppose \eqref{11} is satisfied, then for some $v_1, v_2\in\mathbb{R}^2$ we have
					\begin{gather*}
						u(1)=I_1f+(I-E_{-})v_2,\quad \Lambda u(1)=-I_2(f)+(I-E_+)v_1.
					\end{gather*}
					Multiplying the first equality by $\Lambda$ and subtracting the second equality, we obtain \eqref{12}.
					
					Conversely, suppose that \eqref{12} is true. Then, for some $v_1, v_2\in\mathbb{R}^2$,
					\begin{gather*}
						I_2f+\Lambda  I_1f=-\Lambda(I-E_{-})v_{2}-(I-E_+)v_{1}.
					\end{gather*}
					Define $u(1)=I_1f+(I-E_{-})v_{2}$. Using that $E^2_\pm=E_\pm,\  E_\pm(I-E_\pm)=0$, we verify that \eqref{11} is satisfied.
				\end{proof}
				
				We define the linear map
				\begin{gather*}
					\Phi: C^{1+\gamma}([0,3])\times {\mathcal{D}^\gamma([0,3])}\longrightarrow \frac{\mathbb{R}^2}{\text{ker}(E_{+})+\Lambda\cdot \text{ker}(E_{-})}. 
				\end{gather*}
				by assigning $f$ to the coset $[\Lambda I_1f+I_2f]$.
				
				$\text{ran}(\widetilde{M})\subset \text{ker}(\Phi)$ is already known and moreover \[\text{ran}(\widetilde{M})\subset \text{ker}(\Phi)\cap \{(f_1, f_2):s(s-3)f_2(s)\ \text{vanishes at Kimura tangent points.}\}\] 
				Indeed we show below that they are equal and thus obtain 
				\begin{lemma}\label{phi}
					We have
					\begin{enumerate}
						\item $\text{dim}\ \text{ker}(\Phi)/\text{ran}(\widetilde{M})=|\text{Kimura tangent points}|$,
						\item $\Phi$ is surjective.
					\end{enumerate}
				\end{lemma}
				\begin{proof} 
					We show that
					\begin{gather}\label{range_m}
						\text{ker}(\Phi)\cap \{(f_1, f_2):s(s-3)f_2(s)\ \text{vanishes at Kimura tangent points}\}=\text{ran}(\widetilde{M}).
					\end{gather}
					
					First we assume that there is no Kimura tangent endpoint. Given $f\in \text{ker}(\Phi)$, then for some $u(1)\in\mathbb{R}^{2}$, $E_{+}\Lambda u(1)=-I_1f,\ E_-u(1)=I_2f$. Define $u$ by \eqref{u7}, \eqref{u8}. Then by Lemmas \ref{check1}, \ref{check2}, $u\in C^{2+\gamma}([0,3])\times C^{1+\gamma}([0,3])$ solves $\widetilde{M}u=f$, so that $f\in \text{ran}(\widetilde{M})$, and hence $\text{ran}(\widetilde{M})=\text{ker}(\Phi)$.
					
					If there is at least one Kimura tangent point, then $\text{ker}(E_{+})+\Lambda \text{ker}(E_{-})=\mathbb{R}^2$, so $\text{ker}(\Phi)$ is the whole space $C^{1+\gamma}([0,3])\times {\mathcal{D}^\gamma([0,3])}$. For $f$ in the left hand space of \eqref{range_m}, we define $u$ by \eqref{u7}, \eqref{u8}. Again by Lemmas \ref{check1}, \ref{check2}, such $u$ is the solution of $\widetilde{M}u=f$, and hence we proved \eqref{range_m} in this case.
					
					Next we prove that $\Phi$ is surjective. If there is at least one Kimura tangent endpoint, then $\text{ker}(E_{+})+\Lambda \text{ker}(E_{-})=\mathbb{R}^2$, so the target space of $\Phi$ is $\{0\}$ and $\Phi$ is surjective.
					
					Now we assume that there is no Kimura tangent point. Pick $x\in\mathbb{R}^{2}$. We need to find $f\in C^{1+\gamma}([0,3])\times\mathcal{D}^{\gamma}([0,3])$
					such that
					\begin{gather*}
						\Lambda I_1f+I_{2}f-x\in \text{ker}(E_{+})+\Lambda\cdot\text{ker}(E_{-}).
					\end{gather*}
					Recall that 
					\begin{gather*}
						I_{1}f=E_{+}\Lambda\int_{-1}^{1}Y(s)^{-1}f(s)ds+ \int_{1}^{\infty}e^{(s-1)A_{+}}E_{+}f(s)ds,\\
						I_{2}f=\int_{-\infty}^{-1}e^{(s+1)A_{-}}E_{-}f(s)ds.
					\end{gather*}
					We first take
					\begin{gather*}
						f(t)=0\ \forall\ t<1,\qquad  f(t)=-gA_{+}x\ \forall t>1.
					\end{gather*}
					We could choose an appropriate function $g\in L^{1}([1,\infty)\cap\mathcal{C}_{b}^{0}(\mathbb{R})$ so that
					\begin{gather*}
						I_{1}(f)+\Lambda I_{2}f-x=-(I-E_{+})x
					\end{gather*}
					and therefore $\Phi(f)=[x]$. However $f$ would not be continuous. Thus modifying this idea we set
					\begin{gather*}
						f_{k}(t)=0\ \text{for}\ t\leq 1,\\
						f_{k}(t)=k(1-t)gA_{+}x\ \text{for}\ 1\leq t\leq 1+\frac{1}{k},\\
						f_{k}(t)=-gA_{+}x\ \text{for}\ t\geq 1+\frac{1}{k}.
					\end{gather*}
					We compute
					\begin{gather*}
						I_{1}(f)+\Lambda I_{2}f-x=-(I-E_{+})x+x_{k}
					\end{gather*}
					where $x_{k}\to 0\ \text{as}\ k\to\infty$. Since the projection map is continuous we have $\Phi(f_{k})\to[x]$ and since the image of $\Phi$ is closed we conclude that $[x]\in \text{ran}(\Phi)$. We construct $f$ similarly in all other cases.
				\end{proof}
				
				From Lemma 3.7 we have
				\begin{gather*}
					\text{codim}\ \text{ran}(\widetilde{M})=\text{dim}\ \text{ker}(\Phi)/\text{ran}(\widetilde{M})+\text{dim\ ran}\ \Phi\\
					=|\text{kimura tangent endpoints}|+2-\text{dim}(\text{ker}(E_{+})+\Lambda \text{ker}(E_-)).
				\end{gather*}
				so that
				\begin{align*}
					\text{ind}(L_\gamma)=&\text{ind}(M)=\text{ind}(\widetilde{M})\\
					=&\text{dim}[\text{ker}(E_{-})\cap\Lambda^{-1}\text{ker}(E_{+})]-[2-\text{dim}(\text{ker}(E_{+})\\
					&+\Lambda \text{ker}(E_-))]-|\text{kimura tangent endpoints}|\\
					=&\text{dim}\ \text{ker}(E_{+})+\text{dim}\ \text{ker}(E_{-})-2-|\text{Kimura tangent endpoints}|.
				\end{align*}
				From the definitions of $E_\pm$ we conclude that
				\begin{gather*}
					\text{dim}\ \text{ker}(E_{+})+\text{dim}\ \text{ker}(E_{-})=1+1+|\text{tangent endpoints}|,
				\end{gather*}
				and hence
				\begin{gather*}
					\text{ind}(L_{\gamma})=1+1+|\text{tangent endpoints}|-2-|\text{tangent kimura endpoints}|\\
					= |\text{tangent quadratic endpoints}|\qquad 
					=\ \kappa^{+}+\kappa^{-}.
				\end{gather*}

				\section{Exponential convergence to invariant measures}\label{sec:cv1d}
				In this section we study the convergence rate to the invariant measure in two different cases. 
				First, when there is at least one tangent boundary, then the invariant measure is Dirac measure(s) at the tangent boundary(ies) and quadratic endpoints (if any). To estimate the convergence rate we use Lyapunov functions constructed in Section \ref{sec:lam0estimate}. These Lyapunov functions display different asymptotic  behaviors at different boundary endpoints, and turn out to lie in the function spaces $C(\alpha,\beta)$ (see Definition \ref{c(alpha,beta)}). We show that the growth bound of the generator of $\mathcal{Q}_t$ on $C(\alpha,\beta)$ is not larger than the Lyapunov function rate \eqref{lypunov_rate}, and thus negative. From this, we conclude that the transition probability converges exponentially to the corresponding invariant measure in a Wasserstein distance. The main results are summarized in Theorem \ref{wasser_tant}.

				In the second case where both boundary points are transverse, the invariant measures include a unique invariant measure $\mu$ supported on the whole domain $[0,1]$ and Dirac measures at quadratic endpoints (if any). Motivated by underdamped Langevin dynamics, we wish to prove and then apply a Poincar\'e inequality in $L^2(\mu)$ to show exponential convergence. 
				Using tools developed in \cite{bakry2008rate}, we obtain a (W)-Lyapunov-Poincar\'e inequality; see Definition \ref{def:w-lyap}. The main tools we use are the local Lyapunov function, which is Lyapunov except on a small subset $\mathcal{C}$ and a local Poincar\'e inequality on $\mathcal{C}^c$; see Definition \ref{local_normal_coordinate}. Finally we apply the (W)-Lyapunov-Poincar\'e inequality to show that $\mu$ satisfies the Poincar\'e inequality and for any probability measure $v=h\mu$ with $h\in L^2(\mu)$, $\mathcal{Q}^*_tv$ converges exponentially to $\mu$ in total variation distance. The main results are summarized in Theorem \ref{q_t:l^2}.

				\subsection{Case of one/two tangent boundary points}
				\subsubsection{Function Space $C(\alpha,\beta)$}
				When there is at least one tangent boundary point, we consider a function space isometrically isomorphic to $C^0([0,1])$:
				\begin{definition}\label{c(alpha,beta)}
					For $\alpha, \beta\in\mathbb{R}$, we define the  function space 
					\begin{align*}C(\alpha,\beta):=x^\alpha (1-x)^\beta C^0([0,1])\end{align*}
					equipped with norm
					\begin{gather*}
						||f||_{\alpha,\beta}:=\left|\left|\frac{f}{x^\alpha(1-x)^\beta}\right|\right|_{C^0}.
					\end{gather*}
				\end{definition}
				We will specify $\alpha, \beta$ in different cases later (see Sec.\ref{sec:lam0estimate} for more detail) but list their ranges as: 
				\begin{align*}
					\alpha, \beta=\begin{cases}
						\begin{array}{cc}
							\in(0,1) & \text{Kimura/quadratic tangent}\\
							0 & \text{Kimura transverse}\\
							<0 & \text{quadratic transverse.}
					\end{array}\end{cases}
				\end{align*}
				
				\begin{remark}
					The solution operator $\mathcal{Q}_t$ constructed in Theorem \ref{semigroup} is also the solution operator of the Cauchy problem on $C(\alpha, \beta)$. 
					To see this, the two local operators $\widetilde{Q}^0_t, \widetilde{Q}^1_t$ naturally act on $C(\alpha, \beta)$ and so does the perturbation operator. 
				\end{remark}
				
				We let $A$ be the generator of $\mathcal{Q}_t$ on $C(\alpha,\beta)$. We use the notation $f(x_0)\sim 0$ in $C(\alpha, \beta)$ if $\frac{f}{x^\alpha(1-x)^\beta}(x_0)=0$ and $f\succ 0 (\succeq 0)$ in $C(\alpha,\beta)$ if $\frac{f}{x^\alpha(1-x)^\beta}$ is strictly positive (nonnegative). 
				

				\begin{lemma}[Positive Minimum Principle]\label{positive minimum principle}
					For $\lambda\in\mathbb{R}$, the operator $B:=A-\lambda$ satisfies the positive minimum principle on $C(\alpha, \beta)$, i.e.
					\begin{gather*}
						\text{for every $0\preceq f\in D(A)$ and $x\in [0,1]$, $f(x)\sim 0$ implies $(Bf)(x)\succeq 0$}.
					\end{gather*}
				\end{lemma}
				\begin{proof}
					We say $f\in C^0([0,1])$ is in $D^2([0,1])$ if $f$ is twice differentiable up to Kimura endpoint and if at quadratic endpoint if any, say $x=0$, then $x\partial_xf, x^2\partial_x^2f$ has a continuous limit $0$ at $x=0$. Consider the space
					\begin{gather*}
						D([0,1]):=x^{\alpha^*}(1-x)^{\beta^*}g,\ g\in D^2([0,1]),\\
						\text{where}\ \alpha^*=\alpha, \beta^*=\beta\ \text{if x=0,1 is quadratic and 0 if Kimura}.
					\end{gather*}
					Since $C^2([0,1])$ is dense in $C^0([0,1]$, then $D^2([0,1])\supset C^2([0,1])$ is also dense in $C^0([0,1]$, and hence $D([0,1])$ is dense in $C^{\alpha, \beta}([0,1])$. And for $t>0$, $\mathcal{Q}_t(D([0,1]))\subset D([0,1])$, so by \cite[A-I, Proposition 1.9]{arendt1986one}, $D([0,1])$ is a core for $B$. So it is sufficient to prove the result for $D([0,1])$.
					
					Assume $f\in D([0,1])$ and $f(x_0)=0$. If $x_0$ is an interior point, then $\partial_xf(x_0)=0, \partial^2_xf(x_0)\geq 0$, so $Bf(x_0)\geq 0$. If $x_0$ is Kimura tangent, since $f$ is twice differentiable at $x_0$, $Bf(x_0)=0$. If $x_0$ is Kimura transverse, in this case, $\partial_xf(x_0)\geq 0, x\partial_x^2f(x_0)=0$, and so $Bf(x_0)\geq 0$. 
					
					If $x_0$ is a quadratic endpoint, locally near $0$, $f=x^\alpha g$ for $g\in D^2([0,1])$. As a function $g(x)\geq 0, g(0)=0$ where $x\partial_xg, x^2\partial_x^2g$ have a continuous limit $0$ at $x=0$, so $\underset{x\to 0}{\lim}\frac{Af}{x^\alpha}=0$. Thus, $(Bf)(x_0)\sim 0$.
					
				\end{proof}

				To analyze the semigroup $Q_t$ generated by $A$, we first recall the \textbf{spectral bound $s$} and the \textbf{growth bound $w_1$} for $A$ defined as
				\begin{gather}
					s(A):= \sup\{\text{Re}\ \lambda: \lambda\in\sigma(A)\}\\
					\label{def:growthbound}
					w_1 :=\underset{w\in\mathbb{R}}{\inf}\{||Q_tx||\leq Me^{wt}||x||_{D(A)}, \forall x\in D(A), t\geq 0 \text{\ for suitable}\ M\}
				\end{gather}
				and let
				\begin{gather}\label{lypunov_rate}
					\lambda_{0}:=\underset{\lambda\in\mathbb{R}}{\inf}\ \{Af\leq\lambda f,\  0\prec f\}.
				\end{gather}
				$C(\alpha, \beta)$ equipped the norm $||\cdot||_{\alpha, \beta}$ makes it a Banach lattice as defined in \cite{arendt1986one}.
				A first result relating $s(A)$ and $w_1$ is that \[s(A)=w_{1}(T)\] for $T$ being a positive $C_0$-semigroup on a Banach lattice with generator $A$ (see \cite[Theorem 1.4.1]{van1996asymptotic}). The second result is that:
				\begin{proposition}\label{prop_lambda0}
					$s(A)\leq\lambda_{0}.$
				\end{proposition}
				\begin{proof}
					Our proof is a modification of the proof of \cite[B-II Corollary 1.14]{arendt1986one}. But unlike $C^0([0,1])$, our underlying space $C(\alpha, \beta)$ present different boundary behaviors which should be treated carefully. So we provide a whole proof here.
					
					Given $\lambda>\lambda_0$, by definition of $\lambda_0$, there exists $0<u\in C(\alpha,\beta)$ such that $Au\leq\lambda u$. We define a strict half norm $P_u$ on $C(\alpha,\beta)$: \begin{gather*}
						P_{u}(f)=\underset{x\in[0,1]}{\sup}\frac{f^{+}(x)}{u(x)},\qquad f^+:=\max(f,0).
					\end{gather*} 
					Since $u>0$, $P_u$ is well defined on $C(\alpha,\beta)$. $P_u$ gives rise to a norm
					\begin{align*}||f||_p:=P_u(f)+P_u(-f)\end{align*}
					which is equivalent to the norm on $C(\alpha,
					\beta)$. 
					
					Let $B=\overline{L}-\lambda$, then $Bu\leq 0$.
					We now show $B$ is $P_u$-dissipative (\cite[A-II Definition 2.1]{arendt1986one}), i.e. for all $f\in D(B)$, there exists a $\phi_f$ in the dual space  $C(\alpha,\beta)'$ 
					such that:
					\begin{enumerate}
						\item $(Bf,\phi_f)\leq0$
						\item $\forall g\in C(\alpha,\beta)$, $(g,\phi_f)\leq P_u(g)$. Especially $(f,\phi_f)=P_u(f)$.
					\end{enumerate}

					Fix $f$. If $f\leq 0$, define $\phi_f:=0$. If $f$ is positive at least at one point, denote by $x_0$ a point such that $P_u(f)=\frac{f^+(x_0)}{u(x_0)}$. Now consider $\phi_f\in C(\alpha,\beta)'$ such that  
					\begin{align*}(g,\phi_{f})=\frac{g(x_0)}{u(x_0)}.\end{align*}
					Clearly such $\phi_f$ satisfies the second condition. Next we check the first condition. Let $f\in D(B)$, if $P_{u}(f)=0$, then $f\leq 0$, so $(Bf,\phi_f)=(Bf,0)=0$. If $P_{u}(f)>0$, since $P_{u}(f)=\frac{f(x_0)}{u(x_0)}\geq\frac{f(x)}{u(x)}$, then
					\begin{gather}
						P_{u}(f)\cdot u-f\geq 0,\ (P_{u}(f)\cdot u-f)(x_0)=0.
					\end{gather} 
					By the positive minimum principle Lemma \ref{positive minimum principle},
					\begin{align*}
						Bf(x_0)\leq P_{u}(f)(Bu)(x_0)\leq0
					\end{align*} 
					i.e. $(Bf,\phi_{f})\leq0$. 
					
					Let $P_t$ be the semigroup generated by $B$. We show that $P_t$ is $P_u$-contraction, i.e. $P_u(P_tf)\leq P_u(f), \forall f\in D(B)$.  Let $f\in D(B),\ t>0$, then
					\begin{align*}P_u(f)=(f,\phi_f)=(f-tBf+tBf,\phi_f)\leq(f-tBf,\phi_f)\leq P_u(f-tBf).\end{align*}
					Thus for $\lambda>w_1$, $f\in C(\alpha,\beta)$, $(\lambda-B)R(\lambda,B)f=f$, so $P_u(\lambda R(\lambda,B)f)\leq P_u(f)$,
					then by the formula
					$P_tf=\underset{n\to\infty}{\lim}\left(\frac{n}{t}\cdot R(\frac{n}{t},B)\right)^nf$,
					we have \begin{align*}P_u(P_tf)\leq P_u(f).\end{align*}
					Thus the closure $\overline{B}$ generates a $P_{u}$-contraction semigroup, hence the closure $\overline{A}=\overline{B}+\lambda$ of $A$ generates a positive semigroup of type $w_1(\bar{A})\leq\lambda$. Hence we showed that $\lambda_{0}\geq w_{1}(A)=s(A).$
				\end{proof}

				\subsubsection{Estimation of $\lambda_0$ and rate of convergence}\label{sec:lam0estimate}
				In this part, we show that $\lambda_0<0$, which implies the exponential convergence to $\delta$ measures at tangent boundaries. In order to do this, we construct $u$ that satisfies $Lu<\lambda u$ for some $
				\lambda<0$. The strategy of construction is to first use the behavior of $L$ in the vicinity of two the boundary points to construct $u$ near the boundaries and then find a connecting interior function. 
				\paragraph{Boundary Construction}
				$\mathit{1.\ Quadratic\  endpoint}$.
				In the vicinity of quadratic point, $L=x^2\partial_{xx}+(b(x)+1)x\partial_x$ and recall  ${b_-}=b(0)$. Let \begin{align*}u=Ax^c, A>0.\end{align*}
				Taking $c=\frac{-{b_-}}{2}$,
				\begin{align*}
					\frac{Lu}{u}=(c+b(x))c,
				\end{align*}
				is negative in a neighborhood of 0.
				\bigskip\\
				$\mathit{2. Kimura\ endpoint}.$ 
				In the vicinity of a Kimura point, $L=x\partial_{xx}+b(x)\partial_x$ with $b=b(0)$. Let
				\begin{align*}
					\mathit{u=
						\begin{cases}
							\begin{array}{cc}
								Ax^c,\ 0<c<1 & b=0\\
								A(1-x)^c,\ c>0 & b>0 
							\end{array}
					\end{cases}}.
				\end{align*}
				In both cases
				\begin{gather*}
					\frac{Lu}{u}=(c-1+b(x))cx^{-1}\\
					\frac{Lu}{u}=c(1-x)^{-2}[(c-1)x-b(x)(1-x)]
				\end{gather*}
				are negative in a neighborhood of 0.
				
				\paragraph{Interior Construction}
				Suppose we have constructed $u$ in $[0,x_1]$ and $[x_2,1]$. We now need to construct an interior function $u$ that satisfies the boundary conditions. Inside $[x_1,x_2]$, suppose $L$ has the form \begin{align*}L=\partial_{xx}+b(x)\partial_x,\end{align*}
				where $b(x)$ is smooth on $[x_1,x_2]$.
				Let 
				\begin{align}B(x)=\int_{x_1}^x b(t)dt,\end{align}
				then $B$ is smooth on $[x_1,x_2]$.  We want $Lu$ strictly negative, that is \begin{align*}u_{xx}+b(x)u_x<0.\end{align*} 
				It is equivalent to $e^{B(x)}u_x$ being a decreasing function, i.e. there exists some positive function $f$ on $[x_1,x_2]$, such that
				\begin{align}\label{derivative}
					(e^{B(x)}u_x)'=-f.
				\end{align}
				
				\paragraph{Case I: one tangent, one transverse points}
				Without loss of generality, we assume a tangent boundary at $x=0$. Now by construction of $u$ at both boundaries, $u_x(x_1)>0$, $u_x(x_2)>0$. We first modify the constant $A$ in the construction near $x=1$  such that 
				\begin{align*}
					e^{B(x_2)}u_x(x_2)<e^{B(x_1)}u_x(x_1).
				\end{align*}
				Next, by \eqref{derivative}, $f(x_1)=-e^{B(x_1)}Lu(x_1)>0,\ f(x_2)=-e^{B(x_2)}Lu(x_2)>0$ are two positive fixed constants. Then we can set $f$ to be a positive function with $f(x_1),\ f(x_2)$ fixed and 
				\begin{align*}\int_{x_1}^{x_2}f(t)dt=e^{B(x_1)}u_x(x_1)-e^{B(x_2)}u_x(x_2).\end{align*}
				Such an $f$ guarantees that $e^{B(x)}u_x(x)> e^{B(x_2)}u_x(x_2)>0$ for $x\in [x_1,x_2]$ which guarantees positiveness of $u_x$ and hence positiveness of $u$. 
				\paragraph{Case II: two tangent boundaries} By construction of $u$ at both boundaries, $u_x(x_1)>0,\ u_x(x_2)<0$. We directly get $ e^{B(x_2)}u_x(x_2)<e^{B(x_1)}u_x(x_1)$. 
				
				Denoting $F(x)=\int_{x_1}^x f(t)dt$, then $F(x)$ is a differentiable increasing function on $[x_1,x_2]$ with $F(x_1)=0$. Then by considering derivatives of $u$
				\begin{align*}
					u_x(x)=e^{-B(x)}\left[e^{B(x_1)}u_x(x_1)-F(x)\right]
				\end{align*}
				we know $u$ is first increasing then decreasing on $[x_1,x_2]$, which implies $u$ is lower bounded by $\min(u(x_1),u(x_2))$. Hence we only need to assign a positive $f$ such that 
				\begin{align*}
					\int_{x_1}^{x_2}f(t)dt=e^{B(x_1)}u_x(x_1)-e^{B(x_2)}u_x(x_2).
				\end{align*}

				\begin{theorem}\label{wasser_tant}
					If there is only one tangent endpoint $p$, then for any non-quadratic point $x$, the transition probability $p_t(x,\cdot)$ converges exponentially to $\delta(p)$ in Wasserstein distance. 
					
					If there are two tangent endpoints, then there exists $S_0$, satisfying $S_0(0)=0$, $S_0(1)=1$ and $LS_0=0$, such that for any probability measure $v$, $\mathcal{Q}^*_tv$ converges exponentially to $\delta(0)\int_0^1 v(1-S_0)+\delta(1)\int_0^1 vS_0$ in Wasserstein distance. 
				\end{theorem}
				\begin{proof}
					If there is only one tangent endpoint, say $x=0$, then for any $f\in C^0([0,1])$ with $\text{Lip}(f)\leq 1$, $f$ can be decomposed as
					\begin{gather*}
						f=f_0+f(0), f_0\in C(\alpha,\beta).
					\end{gather*}
					Then by Proposition \ref{prop_lambda0}, 
					\begin{gather*}
						||\mathcal{Q}_tf-f(0)||_{C(\alpha,\beta)}=||\mathcal{Q}_tf_0||_{C(\alpha,\beta)}\leq Me^{\lambda_0t}||f_0||_{C(\alpha,\beta)}
					\end{gather*}
					with some constant $M>0$ independent of $f$. Since $f_0(0)=0, \text{Lip}(f_0)\leq 1$ and $0<\alpha<1, \beta\leq 0$, then $||f_0||_{C(\alpha,\beta)}
					\leq C_{\alpha,\beta}$ for some constant $C_{\alpha,\beta}>0$ only dependent on $\alpha,\beta$. Hence
					\begin{gather}\label{q_t}
						|\mathcal{Q}_tf(x)-f(0)|\leq MC_{\alpha,\beta}e^{\lambda_0t}x^\alpha(1-x)^\beta.
					\end{gather}
					Notice that only when $x=1$ is a quadratic endpoint do we have $\beta<0$. Hence for any non-quadratic point $x$, the transition probability $q_t(x,\cdot)$ converges to $\delta(0)$ in the sense of Wasserstein distance at an exponential rate. This convergence is uniform on any compact interval away from quadratic transverse point. If $x=1$ is a quadratic point, then $q_t(x,\cdot)=\delta(1)$ for $\forall t>0$.
					
					If both endpoints are tangent, by Theorem \ref{maximum_prin}, the kernel of $L$ is in the linear space of $\{1,S\}$. Then we can take $S_0$ as a linear combination of $1$ and $S$, such that $S_0(0)=0$ and $S_0(1)=1$.  For any $f\in C^0([0,1])$ with $\text{Lip}(f)\leq 1$, we decompose this as 
					\begin{gather*}
						f=f_0+(1-S_0)f(0)+S_0f(1), \quad f_0\in C(\alpha,\beta).
					\end{gather*}
					Again, by Proposition \ref{prop_lambda0}, 
					\begin{gather*}
						||\mathcal{Q}_tf-(1-S_0)f(0)-S_0f(1)||_{C(\alpha,\beta)}=||\mathcal{Q}_tf_0||_{C(\alpha,\beta)}\leq Me^{\lambda_0 t}||f_0||_{C(\alpha,\beta)}
					\end{gather*}
					with $M>0$ independent of $f$. Since $0<\alpha,\beta<1, f_0(0)=f_0(1)=0$ and $\text{Lip}(f_0)\leq 1$, then \begin{gather*}
						||\mathcal{Q}_tf-(1-S_0)f(0)-S_0f(1)||_0\leq M_{\alpha,\beta}||\mathcal{Q}_tf-(1-S_0)f(0)-S_0f(1)||_{C(\alpha,\beta)},\\
						||f_0||_{C(\alpha,\beta)}\leq C_{\alpha,\beta}
					\end{gather*}
					for some constant $M_{\alpha,\beta}, C_{\alpha,\beta}>0$ only dependent on $\alpha,
					\beta$. Hence,
					\begin{gather*}
						||\mathcal{Q}_tf-(1-S_0)f(0)-S_0f(1)||_0\leq MC_{\alpha,\beta}M_{\alpha,\beta}e^{\lambda_0t},
					\end{gather*}
					so that the transition probability $q_t(x,\cdot)$ converges uniformly on $[0,1]$ to $(1-S_0(x))\delta(0)+S_0(x)\delta(1)$ at an exponential rate in Wasserstein distance. Hence starting from any probability measure $v$, $\mathcal{Q}_t^*v$ converges exponentially to $\delta(0)\int v(1-S_0)+\delta(1)\int vS_0$ in Wasserstein distance.
				\end{proof}

				\begin{remark}
					Consider as a the first example the case with one tangent boundary point, say $x=0$, quadratic tangent while $x=1$ is Kimura transverse. For instance,
					\begin{gather*}
						L=x^2(1-x)\partial_{xx}+bx(a-x)\partial_{x}\ \text{on [0,1]},\ ab<1,\ b(a-1)<0.
					\end{gather*}
					Consider $f=x^c$. If we choose $0<c<1-ab$, then
					\begin{align*}
						\frac{Lf}{f}=\frac{Lx^c}{x^c}=\left[(1-x)(c-1)+b(a-x)\right]c\leq max(c-1+ab,b(a-1))c<0.
					\end{align*}
					So 
					\begin{gather*}
						\lambda_0\leq \underset{0<c<1-ab}{\min}\max(c-1+ab,b(a-1))c<0.
					\end{gather*}

					As a second example, consider two tangent boundary points , for instance
					\begin{align*}
						L=x^2(1-x)\partial_{xx}\ \text{on}\ [0,1].
					\end{align*}
					In this case, we should expect $p(t,x,\cdot)\to (1-S_0)\delta_0+S_0\delta_1$, where $S_0$ in this case is $x$. By setting $f=x^c(1-x),\ 0<c<1$, then 
					\begin{align*}
						Lf=(c(c-1)-c(c+1)x)x^c(1-x)\leq\lambda f.
					\end{align*}
					Here, $\lambda=c(c-1)<0$. So
					\begin{gather*}
						\lambda_0\leq\underset{0<c<1}{\min}c(c-1)= -\frac{1}{4}.
					\end{gather*}
				\end{remark}

				\subsection{Case with two transverse boundaries}
				Recall that we have constructed the invariant measure $\mu$ when both endpoints are transverse in Section 3.2. We first introduce a (W)-Lyapunov-Poincar\'e inequality that will be used to analyze the long time behavior of the transition probability.
				\subsubsection{(W)-Lyapunov-Poincar\'e inequality} 
				Consider $U(z)=\int_0^z b(s)ds$ with $b(s)$ negative when $s\to-\infty$ and positive when $s\to\infty$. Then $U$ grows linearly to $+\infty$ near $z=\pm\infty$.  $L$ can be taken as
				\begin{align*}
					L_z=\frac{1}{2}\partial_{zz}-\nabla U(z)\partial_z=-\frac{1}{2}\partial_z^*\partial_z
				\end{align*}
				with $\partial^*_z=\partial_z+2\nabla U(z)$ on a probability space $(X,\mu)$, with $\mu$ the invariant measure \begin{align}\label{eq:mutransverse} \mu(dz)=\frac{e^{-2U(z)}}{Z}dz\end{align}
				where the normalizing constant $Z<\infty$ due to the linear growth of $U$ . 
				The main advantage of switching to $L^2(\mu)$ is
				$L$ is self adjoint on $L^2(\mu)$. Indeed
				\begin{align*}\int_{\mathbb{R}}(\partial_z\varphi)\phi d\mu=\int_{\mathbb{R}}\varphi(-\partial_z \phi+2\phi\nabla U(z))d\mu=\int_{\mathbb{R}}\varphi\partial^*_z\phi d\mu\end{align*}
				so that
				\begin{align*}(\varphi,L\phi)_{L^2(\mu)}=(L\varphi,\phi)_{L^2(\mu)}.\end{align*}

				For some function $W\in D(L)$ with $W\geq 1$, let $I_W(t)=\int Q^2_tf Wd\mu$. Then \cite{bakry2008rate}
				\begin{align}\label{i_w'(t)}
					I'_W(t)=-\int W|\nabla (Q_tf)|^2d\mu+\int W LQ^2_tfd\mu.
				\end{align} This leads to the definition $\textbf{(W)-Lyapunov-Poincar\'e inequality}$: 
				\begin{definition}\label{def:w-lyap}
					We say that $\mu$ satisfies a (W)-Lyapunov-Poincar\'e inequality, if there exists $W\in D(L), W\geq 1$ and a constant $C_{LP}$ such that for all nice $f$ with $\int fd\mu=0$,
					\begin{align}
						\int f^2Wd\mu\leq C_{LP}\int \left(W|\nabla f|^2- W Lf^2\right)d\mu.
					\end{align}
				\end{definition}
				

				
				\paragraph{Local {Poincar\'e} inequality}
				In the transverse case, we can not find a global Lyapunov function. We have to weaken the condition to define a  $\textbf{local Lyapunov function}$ $V$ as
				\begin{align}\label{func-lyapunov}
					LV\leq -\alpha V+\beta\mathbf{1}_{\mathcal{C}},\ V\geq1
				\end{align}
				for some set $\mathcal{C}$. 
				
				\begin{definition}[Local {Poincar\'e} inequality]\label{local-Poincare-ineq}
					Let ${\Omega}$ be a subset of the whole space $X$ (we will use $\mathcal{C}\subset {\Omega}$). We say that $\mu$ satisfies a
					local {Poincar\'e} inequality on ${\Omega}$ if there exists some constant $\kappa_{\Omega}$ such that for all nice $f$ with
					$\int_X fd\mu=0$,
					\begin{align}
						\int_{\Omega} f^2d\mu\leq \kappa_U\int_X|\nabla f|^2d\mu+\frac{1}{\mu({\Omega})}\left(\int_{\Omega} fd\mu\right)^2.
					\end{align}
				\end{definition}
				
				We now show the:
				\begin{proposition}\label{prop-w-lya}
					Assume that there is some local Lyapunov function $V\geq 1$ under some set $\mathcal{C}$ such that \eqref{func-lyapunov} holds, and $\mu$ satisfies a local {Poincar\'e} inequality on ${\Omega}\supset \mathcal{C}$ with moreover
					\begin{align}
						\beta\mu({\Omega}^c)<\alpha\mu({\Omega}).
					\end{align}
					Then we can find $\lambda\geq 0$ such that if $W=V+\lambda$, then $\mu$ satisfies a (W)-Lyapunov-{Poincar\'e} inequality.
				\end{proposition}
				\begin{proof}
					Multiply \eqref{func-lyapunov} by $f^2$ and integrate to get
					\begin{align*}
						\int_ X f^2LVd\mu\leq -\alpha\int_X f^2Vd\mu+\beta\int_Cf^2d\mu.
					\end{align*}
					Let $\int_X fd\mu=0$. The local {Poincar\'e} inequality implies 
					\begin{align*}
						\int_{\Omega} f^2d\mu & \leq\kappa_{\Omega}\int_X|\nabla f|^2d\mu+\frac{1}{\mu({\Omega})}\left(\int_{\Omega} fd\mu\right)^2\\
						& \leq \kappa_{\Omega}\int_X|\nabla f|^2d\mu+\frac{1}{\mu({\Omega})}\left(-\int_{{\Omega}^c} fd\mu\right)^2 \\
						&\leq \kappa_{\Omega}\int_X|\nabla f|^2d\mu+\frac{\mu({\Omega}^c)}{\mu({\Omega})}\left(\int_{{\Omega}^c} f^2d\mu\right),
					\end{align*}
					so that
					\begin{align*}
						\int_Xf^2\left(LV+\alpha V\right)d\mu&\leq \beta\int_Cf^2d\mu  \leq \beta\int_{\Omega}f^2d\mu
						\\&\leq \beta\kappa_U\int_X|\nabla f|^2d\mu+\frac{\beta\mu({\Omega}^c)}{\mu({\Omega})}\left(\int_{{\Omega}^c} f^2Vd\mu\right).
					\end{align*}
					We re-organize the inequality as,
					\begin{align*}
						\alpha\left(1-\frac{\beta\mu({\Omega}^c)}{\alpha\mu({\Omega})}\right)\int_Xf^2Vd\mu & \leq \beta\kappa_{\Omega}\int_X|\nabla f|^2d\mu-\int_Xf^2LVd\mu\\
						& \leq\int_X|\nabla f|^2(V+\lambda)d\mu-\int_Xf^2L(V+\lambda)d\mu\\
						& =\int_X|\nabla f|^2(V+\lambda)d\mu-\int_X Lf^2(V+\lambda)d\mu.
					\end{align*}
					Here we take$\ \lambda=(\beta\kappa_U-1)_+$ so that $\beta\kappa_{\Omega}\leq V+\lambda$. 
					Since,
					\begin{align*}\int_X f^2(V+\lambda)d\mu\leq(1+\lambda)\int_X f^2Vd\mu,\end{align*}
					then
					\begin{align*}
						\alpha\left(1-\frac{\beta\mu({\Omega}^c)}{\alpha\mu({\Omega})}\right)\frac{1}{1+\lambda}\int_X f^2(V+\lambda)d\mu\leq\int_X\left(|\nabla f|^2(V+\lambda)-Lf^2(V+\lambda)\right)d\mu
					\end{align*}
					and $\frac{1}{C_{LP}}=\alpha\left(1-\frac{\beta\mu({\Omega}^c)}{\alpha\mu({\Omega})}\right)/(1+\lambda)$.
				\end{proof}

				\subsubsection{Long time behaviors}
				Let us come back to the operator $L=\frac{1}{2}\Delta-\nabla U\cdot\nabla$ on $(X,\mu)$.
				We first construct the local Lyapunov function $V$ with $\mathcal{C}$ satisfying \eqref{func-lyapunov}. 
				\begin{lemma}
					There exists a local Lyapunov function.
				\end{lemma}
				\begin{proof}
					Let $X=[x_l,x_r],\ -\infty\leq x_l<x_r\leq\infty$ and
					\begin{align*}V=\begin{cases}
							\begin{array}{cc}
								e^{U} & x_{l}=-\infty\\
								2-(x(z)-x_{l}) & x_{l}>-\infty.
					\end{array}\end{cases}\end{align*}
					We define $V$ similarly in the vicinity of $x_r$.  We check that $LV\leq -\alpha V,\ V>1$ for some $\alpha>0$ in a neighborhood $x_l$:
					\begin{enumerate}
						\item When $x_l=-\infty$, for $V=e^{U}$,
						\begin{align*}
							LV=V\Big(\frac{1}{2}\Delta U-\frac{1}{2}|\nabla U|^2\Big).
						\end{align*}
						By the transverse condition, $\underset{z\to -\infty}{\lim}\nabla U=-b_{-}<0$, which implies $\underset{z\to\infty}{\lim}\Delta U=0$. So 
						\begin{align*}\underset{z\to -\infty}{\lim}\Big(\frac{1}{2}\Delta U-\frac{1}{2}|\nabla U|^2\Big)=-\frac{1}{2}b_-^2<0.\end{align*}
						Since $\underset{z\to -\infty}{lim}U(z)=\infty$, for $0<\alpha<\frac{b_-^2}{2}$, we have $LV\leq -\alpha V,\ V\geq 1$ in a neighborhood of $x_l$.
						
						\item When $x_l>-\infty$, for $V=2-(x(z)-x_l)$, 
						\begin{align*}LV=L_x(2-x)=-b(x).\end{align*}
						By the transverse condition, $b_{-}>0$, so $LV(0)<0$ and since $V(0)=2>0$, $LV\leq -\alpha V,\ V\geq 1$  for some $\alpha>0$ in a neighborhood of $x_l$.
						
						In all cases, $LV\leq -\alpha V,\ V>1$ for some $\alpha>0$ in a neighborhood $x_l, x_r$.
					\end{enumerate}
					
					After constructing $V$ in a neighborhood of $x_l$ and $x_r$, we extend it to a complement $\mathcal{C}$ of these two neighborhoods as a twice differentiable function with $V\geq 1$. Since $\mathcal{C}$ is compact, $LV, V$ are bounded on $\mathcal{C}$, and we can find some $\beta>0$ large enough that $V$ is a Lyapunov function for $\mathcal{C}$:
					\begin{gather*}
						LV\leq -\alpha V + \beta\mathbf{1}_{\mathcal{C}}.
					\end{gather*}
				\end{proof}
				
				Next we verify that $\mu$ satisfies a local {Poincar\'e} inequality \eqref{local-Poincare-ineq}. Indeed by \cite[Theorem 1]{Poincare_interval}, if $\mu$ is bounded and restricted to a Euclidean ball of radius $R$, we know that it satisfies the {Poincar\'e} inequality, hence the local {Poincar\'e} inequality on it.
				
				We now state our main result of the section:
				\begin{theorem}\label{q_t:l^2}
					The invariant measure $\mu$ in \eqref{eq:mutransverse} satisfies a Poincar\'e inequality. For $f\in L^2(\mu)$,
					\begin{gather}\label{no_good_name}
						||\mathcal{Q}_tf-\int f\mu||_{L^2(\mu)}\leq e^{-\frac{t}{2C_{LP}}}||f-\int f\mu||_{L^2(\mu)}.
					\end{gather}
					For any probability measure $v=h\mu$ with $h\in L^2(\mu)$,
					\begin{gather*}
						||\mathcal{Q}_t^*v-\mu||_{\text{TV}}\leq e^{-\frac{t}{2C_{LP}}}||h-1||_{L^2(\mu)}.
					\end{gather*}
				\end{theorem}
				\begin{proof}
					By construction, $\mathcal{C}$ is a compact set in the interior of $X$. We choose a larger set $U$ with $\mathcal{C}\subset U\subset int(X)$ such that $\beta\mu(U^c)<\alpha\mu(U)$. $\mu$ is bounded on $U$, so $\mu$ satisfies a local {Poincar\'e} inequality on $U$.
					Applying Proposition \ref{prop-w-lya},  $\mu$ satisfies a (W)-Lyapunov-{Poincar\'e} inequality. Thus by \eqref{i_w'(t)}, for all $f$ such that $\int f^2W\mu<\infty$ and $\int f\mu=0$, 
					\begin{gather*}
						\int(\mathcal{Q}_tf)^2W\mu\leq e^{-\frac{t}{C_{LP}}}\int f^2W\mu.
					\end{gather*}
					Since $W\in L^1(\mu)$, by \cite[Corollary 3.4]{bakry2008rate}, $\mu$ satisfies a Poincar\'e inequality with Poincar\'e constant equal to $C_{LP}$. Hence \eqref{no_good_name} is derived by Poincar\'e inequality. 
					
					For the second assertion, we use the symmetry property $\mathcal{Q}_t^*(h\mu)=(\mathcal{Q}_th)\mu$ and \eqref{no_good_name} to derive that 
					\begin{gather*}
						||\mathcal{Q}^*_tv-\mu||_{\text{TV}}=||\mathcal{Q}_th-1||_{L^1(\mu)}\leq ||\mathcal{Q}_th-1||_{L^2(\mu)}\leq e^{-\frac{t}{2C_{LP}}}||h-1||_{L^2(\mu)}.
					\end{gather*}
				\end{proof}

				\begin{remark}
					Regarding logarithmic Sobolev inequalities, which are stronger than {Poincar\'e} inequalities, we may use a criterion on a ball of radius $R$ (\cite[Proposition 2.6]{logarithmic}) to obtain when both of the endpoints are of Kimura type, that  $\mu$ satisfies a logarithmic Sobolev inequality.
				\end{remark}

				\section{Long time behaviour in two dimension}\label{sec:2d}
				In this section we analyze the long time behavior of transition probability on a two-dimensional manifold with corners. We refer to \cite{chen2022mixed} for an introduction to 2-dimensional manifold with corners in our context and  \cite{DJ} for a more complete introduction.
				
				Let $P$ be a paracompact Hausdorff topological space. A chart of $p\in P$ is a pair $(\mathcal{U}_p, \psi_p)$ where $\mathcal{U}_p$ is a neighborhood of $p$ and $\psi_p$ is a homeomorphism with $\psi_p(p)=0$ from $\mathcal{U}_p$ to a neighborhood of $0$ in $\mathbb{R}^l_+\times\mathbb{R}^{2-l}$ for some $l\in \{0,1,2\}$. Two charts $(\mathcal{U}_p, \psi_p), (\mathcal{U}_q, \psi_q)$ are said to be compatible if 
				\[\psi_p\circ \psi_q^{-1}: \psi_q(\mathcal{U}_p\cap\mathcal{U}_q)\longrightarrow\psi_p(\mathcal{U}_p\cap\mathcal{U}_q)\]
				is a diffeomorphism. The codimension $l$ is well defined for $p\in P$. We say that the point $p$ is an interior point if $l=0$, an edge
				point if $l=1$, a corner if $l=2$. A two dimensional manifold with corners is paracompact Hausdorff topological space equipped with a maximal compatible atlas. 
				
				\subsection{Lyapunov function}
				We make the following assumption:
				\begin{assumption}\label{assp:onetangentedge}
					For $L$ on a 2 dimensional compact manifold with corners $P$, there is exactly one tangent edge $H$, and when restricted to $H$, $L|_H$ is transverse to both boundary points.
				\end{assumption}

				We say that $V$ is a Lyapunov function if $V$ is strictly positive except on $H$, $V|_H\equiv 0$ and for some $\lambda_0<0$,
				\begin{align}\label{lv<v}
					LV\leq\lambda_0 V.
				\end{align}
				
				We make the following assumption:
				\begin{assumption}\label{assp:rhoexistence}
					There is a twice differentiable function $\rho(p)$ such that
					\begin{enumerate}
						\item[a.] $\rho>0$ except on $H$, $\rho|_H\equiv 0$;
						\item[b.]\label{condition_b} $\nabla\rho\neq 0$ nowhere vanishing;
						\item[c.] For $p$ in an edge $E$, if $\nabla_{E}\rho(p)=0$, then $\rho\equiv c$ in a relative neighborhood of $p$ along $E$;
						\item[d.]There is no local minimum point of $\rho$ other than on $H$.
					\end{enumerate}
				\end{assumption}
				The third assumption ensures that $\rho$ is constant or strictly monotonic along any edge $E$. In the presence of such $\rho$, there exists a stratification of $P$ so that:
				\begin{enumerate} 
					\item $P$ is covered by the layers:
					\begin{gather}\label{cover}
						P=\underset{i=0}{\overset{n}{\bigcup}}P_{[k_i,k_{i+1}]}:=\underset{i=0}{\overset{n}{\bigcup}} \{p\in P: \rho(p)\in [k_i,k_{i+1}]\}, \\ 0=k_0<k_1<\cdot\cdot\cdot k_n=\underset{P}{\max} 	\ \rho. \nonumber
					\end{gather}
					\item Each layer $P_{[k_i,k_{i+1}]}$ is covered by a finite collection of closed sets ${U}_{i,p_j}$ , i.e. 
					\begin{gather}\label{layer_b}
						P_{[k_i,k_{i+1}]}=\bigcup_{j}U_{i,p_j},
					\end{gather}
					where $p_j$ is a point in $U_{i,p_j}$ and the range of $\rho$ in each $U_{i,p_j}$ is $[k_i, k_{i+1}]$. 
				\end{enumerate}
				We leave the construction of the stratification to Lemma \ref{lem:stratification} in the Appendix. We then have the following result:
				\begin{theorem}\label{thm:lyapunov}
					There exists a Lyapunov function $V$ on $P$.
				\end{theorem}
				\begin{proof}
					We first analyze the local behavior of $\rho, L$. Note that for
					\begin{gather*}
						L=A(x,y)\partial_{xx}+B(x,y)\partial_{xy}+C(x,y)\partial_{yy}+D(x,y)\partial_x+E(x,y)\partial_y,
					\end{gather*}
					under the new coordinates $(\eta,\rho)$, the ordinary differential $\rho$-part $L_\rho$ is 
					\begin{gather}\label{l_rho}
						L_\rho=(A\rho_{xx}+B\rho_{xy}+C\rho_{yy})\partial_{\rho\rho}+(L\rho)\partial_\rho.
					\end{gather}
					Given any point $p\in P$, we discuss the local property of $L$ at point $p$ subject to the following cases.
					\paragraph{1. $p$ is an interior point} 
					There exists a closed set including $p$ equipped with coordinates $(x,\rho)$ that is diffeomorhic to the rectangle $R:=[0,1]\times [\rho(p),\rho(p)+\epsilon]$. In this case, $L_\rho$ is elliptic of the form
					\begin{gather}\label{l_rho_interior}
						L_\rho= a(x,\rho)[\partial_{\rho\rho}+b(x,\rho)\partial_\rho], (x,\rho)\in R,
					\end{gather}
					where $a(x,\rho)>0$.
					
					\paragraph{2. $p$ is an edge point} Under local adapted coordinates $(x,y)$,
					\begin{enumerate}
						\item[a.] if $\rho\equiv\rho(p)$ along the edge $y=0$, then $(x,\rho)$ are local adapted coordinates. There is a neighborhood of $p$ that is diffeomorphic to the rectangle
						\[R:=[0,1]\times [\rho(p)-\epsilon,\rho(p)].\] 
						$L_\rho$ is degenerate at $\rho(p)$ of the form: for $(x,\rho)\in R$, \begin{gather}
							\label{l_rho_edge1}
							\textbf{Kimura}: L_\rho= a(x,\rho)\left[(\rho(p)-\rho)\partial^2_\rho+c(x,\rho)\partial_\rho\right]\\
							\label{l_rho_edge2}
							\textbf{quadratic}: L_\rho= a(x,\rho)\left[(\rho-\rho(p))^2\partial_\rho^2+d(x,\rho)(\rho(p)-\rho)\partial_\rho\right]
						\end{gather}
						where $a(x,\rho)>0$.
						
						\item[b.] If the gradient of $\rho$ along the edge $\nabla_x\rho(p)\neq 0$, then $(y,\rho)$ are local adapted coordinates. There is a neighborhood of $p$ that is diffeomorphic to $R:=[0,1]\times [\rho(p),\rho(p)+\epsilon]$ for some $\epsilon>0$. In this case $L_\rho$ is elliptic of the same form with (\ref{l_rho_interior}) in $R$.
					\end{enumerate}

					\paragraph{3. $p$ is a corner} Under local adapted coordinates $(x,y)$, 
					\begin{enumerate}
						\item[a.] if along one (e.g. $x$) edge $\nabla_x\rho(p)=0$, then $\rho\equiv\rho(p)$ in a relative open neighborhood of $p$ on the $x$-edge. Since $\nabla_y\rho(p)\neq 0$, then by rescaling $x$, $(x,\rho)$ are also local adapted coordinates at $p$;
						this case is then the same as in  (\ref{l_rho_edge1}), (\ref{l_rho_edge2}). See the first picture in Figure ~\ref{fig: fig_corner}.
						
						
						\item[b.]If $\nabla_x\rho(p), \nabla_y\rho(p)\neq 0$, then $(x,\rho)$ are not adapted coordinates. In this case, there exists a neighborhood of $p$ that is diffeomorphic to an irregular area inside a rectangle
						\[T:=\{(x,\rho):\rho\geq\rho(x)|_{y=0}\}\subset [0,1]\times [\rho(p)-\epsilon_1,   \rho(p)+\epsilon_2].\]
						See the second and third pictures in Figure \ref{fig: fig_corner}. $L_\rho$ is then degenerate when $x=0$ and elliptic when $x\neq 0$. 

						If $p$ is neither a local maximum nor minimum, $T$ is made up of a rectangle and an irregular area $T_1$, we can extend the coefficients of $L_\rho$ to the rectangle supplemented by the dashed lines. Then in these two rectangles, $L_\rho$ are both degenerate when $x=0$ and elliptic when $x\neq 0$.
					\end{enumerate}
					\begin{figure}[htbp]
						\centering
						\includegraphics[width=1.0\linewidth]{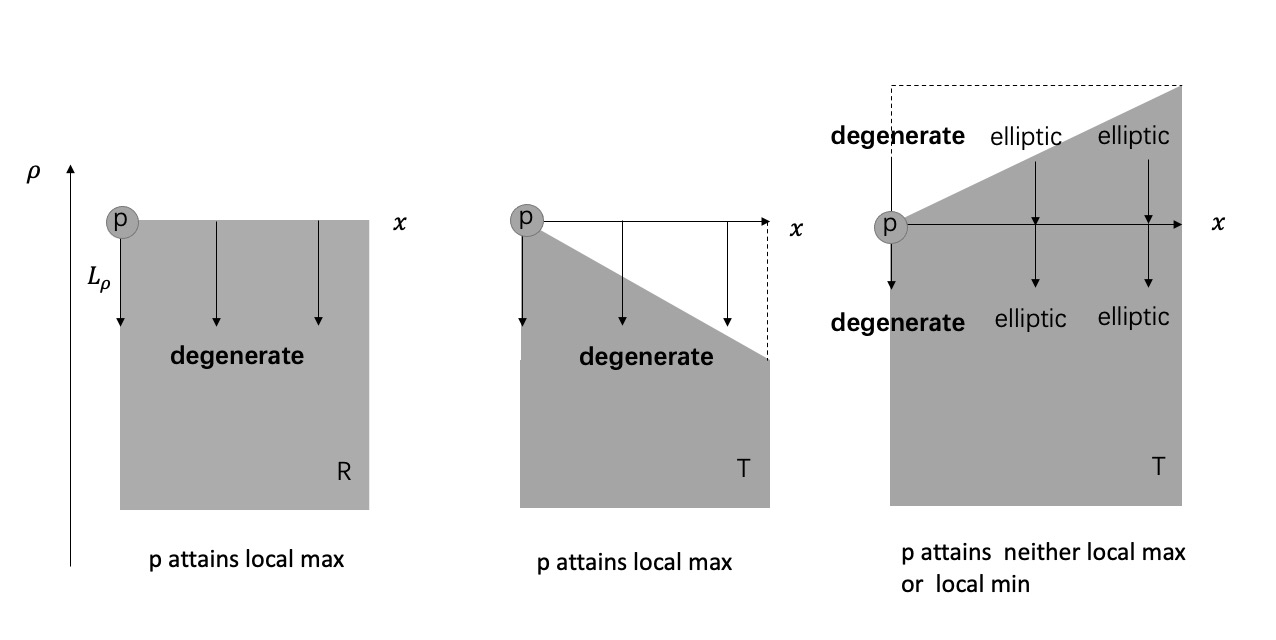}
						\caption{$L_\rho$ at a corner}
						\label{fig: fig_corner}
					\end{figure}
					
					We now start the construction of $V$ iteratively in each layer. Such $V$ generally depends only on the variable $\rho$, i.e., {$V(p)=V(\rho(p))$}, except near the corners.
					\paragraph{Case $i=0$} $P_{[0,k_1]}$ is covered by finite rectangles $R:=[0,1]\times [0,k_1]$ where $L_{\rho}$ takes the form 
					\begin{gather}
						\label{rho_kimura}\textbf{Kimura}: L_{\rho}=a(x,\rho)\left(\rho\partial_\rho^2+b(x,\rho)\partial_\rho\right)\\
						\label{rho_quadratic}\textbf{quadratic}: L_\rho= a(x,\rho)(\rho^2\partial_\rho^2+c(x,\rho)\rho\partial_\rho)
					\end{gather}
					where $a(x,\rho)>0, b(x,\rho)<1, c(x,\rho)<1$ since $H$ is a tangent edge. We take $V=\rho^\nu$, where $\nu$ depends on the considered case. In case (\ref{rho_kimura}), we choose $0<\nu<\underset{(x,\rho)}{\min}(1-b(x,\rho))$ so that $L_\rho\rho^\nu=a\nu(\nu+b-1)\rho^{\nu-1}\leq -\lambda\rho^\nu$ for some $\lambda>0$. In case (\ref{rho_quadratic}), we choose $0<\nu< \underset{(x,\rho)}{\min} (1-c(x,\rho))$ so that $L_\rho\rho^\nu=a\nu(\nu+c-1)\rho^\nu\leq -\lambda\rho^\nu$ for some $\lambda>0$. 
					
					\paragraph{Case $i\geq 1$} $P_{[k_i, k_{i+1}]}$ is then covered by four kinds of
					disjoint collections of open set $R$ as we analyzed before, namely:
					\begin{enumerate}
						\item[I.] $R$ does not contain any corner or local maximum point, $L_\rho$ is elliptic for $\forall x$ of the form \eqref{l_rho_interior};
						\item[II.] $R$ contains an edge and $\rho$ is constant along this edge, so $L_\rho$ takes the form \eqref{l_rho_edge1}, \eqref{l_rho_edge2};
						\item[III.] $R$ contains a corner that is a strict local maximum point of $\rho$,
						\item[IV.] $R$ contains a corner that is neither a local maximum point nor local minimum point of $\rho$. 
					\end{enumerate}
					At this layer $P_{[k_i, k_{i+1}]}$, $R_{II}$ and $R_{III}$ are both disconnected from the other three types.  $R_I$ may only be connected with $R_{IV}$. There maybe more than one $R$ of type II, but since $L_\rho$ takes a uniform form in these $R_{II}$, in our construction below we may regard these $R_{II}$ as a uniform $R$ and for convenience of notation we still let $x\in [0,1]$. $R_{III}$ will be treated in this way as well. For $R_{I}$, we can merge these $R_I$ as a uniform one as well. If there exists $R_{IV}$, we construct $V$ separately in $R_{IV}$ so that it may be patched with $V$ in $R_I$.
					
					Thus, it suffices to construct $V$ in these four cases separately. Since $k_i>0$ when $i\geq 1$, we assume that  $\rho(p)=1$.
					
					\paragraph{Case I:} Fix an arbitrary $x_0$,  we choose 
					$f(\rho)$ so that
					\begin{gather*}
						f'(\rho)<b(x_0,\rho)-b(x,\rho).
					\end{gather*}
					Starting from $\rho=1$, we let, 
					$V_\rho=\exp(f(\rho)-\int^\rho b(x,z)dz)$, then $\frac{L_\rho V_\rho}{V_\rho}=[f'-b(x_0,\rho)+b(x,\rho)]<0$. 
					\paragraph{Case II:}In the Kimura case \eqref{l_rho_edge1}, we fix $x_0$ so that $c(x_0, 1)=\underset{x\in[0,1]}{\max}\ c(x,1)$. We choose $f(\rho)$ so that
					\begin{gather*}
						f'(\rho)<\frac{c+c(x_0,\rho)-c(x,\rho)}{1-\rho}, c=-c(x_0,1)>0
					\end{gather*}
					on $[1-\epsilon, 1]$. Denote $H^x(\rho)=\text{exp}\left[\int^\rho \frac{c(x,z)}{1-z}dz\right]\sim(1-\rho)^{-c(x,1)}$,
					and let 
					\[V_\rho=\frac{(1-\rho)^ce^{f(\rho)}}{H^{x_0}(\rho)}.\]
					Then, $\frac{L_\rho V_\rho}{V_\rho}=[(1-\rho)f'-c-c(x_0,\rho)+c(x,\rho)]<0$.
					
					In the quadratic case (\ref{l_rho_edge2}), we fix $x_0$ so that $d(x_0, 1)=\underset{x\in[0,1]}{\max}\ d(x,1)$. We choose $f(\rho)$ so that
					\begin{gather*}
						(1-\rho)f'-(d-1)-d(x_0,\rho)+d(x,\rho)<0, d-1=-d(x_0,1)-1>0
					\end{gather*}
					on $[1-\epsilon, 1]$. Denote $H^x(\rho)=\text{exp}\left[\int^\rho \frac{d(x,z)}{1-z}dz\right]\sim(1-\rho)^{-d(x,1)}$, let \[V_\rho=\frac{(1-\rho)^{d-1}e^{f(\rho)}}{H^{x_0}(\rho)}.\] 
					Then, $\frac{L_\rho V_\rho}{V_\rho}=(1-\rho)[(1-\rho)f'-(d-1)-d(x_0,\rho)+d(x,\rho)]<0$.
					
					\paragraph{Case III:}If the corner $p$ is a strict local maximum point, we can extend the coefficients of $L_\rho$ to the rectangle which is supplemented by the dashed lines so that $L_\rho$ is has the same degree of degeneracy at $\rho=\rho(p)$. 
					
					\paragraph{Case IV:}
					In this case, if $L$ has the same degeneracy type towards two edges intersecting at the corner $p$. Let $U$ be a neighborhood of $p$ inside $T$ and introduce new coordinates $(x,z)$ so that $z=\rho$ in $T\setminus U$, $z_\rho>0$ and $p$ attains local maximum of $z$. So in the $(x,z)$-coordinate, $L_z$ has the same kind of degeneracy as $L_\rho$ in an irregular area $T'$. We extend the coefficients of $L_z$ to a rectangle so that $L_z$ takes the form (\ref{l_rho_edge1}), (\ref{l_rho_edge2}). Taking the solutions $V$ to these two forms which we already constructed, we see $V$ is infinity at the corner of two quadratic edges. 
					
					If $L$ has different types of degeneracy towards the two edges, then it takes the form
					\begin{gather}
						L=\left[xa(x,\rho)+(\rho-1)^2\right]\partial_{\rho\rho}+xb(x,\rho)\partial_{xx}+c(x,\rho)\partial_\rho+d(x,y)\partial_x+x(\rho-1)e(x,y)\partial_{x\rho}
					\end{gather}
					on $[0,1]\times [0,1]$, $\rho\in [0,1]$, where $a,b,c,d>0$. 
					
					When $x=0$, $L^0_\rho=(1-\rho)^2\partial_{\rho\rho}+c(x,\rho)\partial_\rho$. Clearly a decreasing linear function $V$ would satisfy $L^0_\rho V<0$. But such $V$ has negative derivative, which is not desired. So our strategy is to first construct $f_0(\rho), f_1(\rho)$ at $x=0,1$ so that 
					\[L^0_\rho f_0, L^1_\rho f_1<0,\ f'_1(\rho)>0.\] 
					We then patch them together with the desired boundary conditions. We choose two non-negative functions $\chi_0(x), \chi_1(x)$ so that
					\begin{gather*}
						\chi_0(0)>0, \chi_0=0\ \text{in}\  [1-\epsilon, 1],\\
						\chi_1(0)=0, \chi_1=\chi(1)\ \text{in}\  [1-\epsilon, 1].
					\end{gather*}
					Suppose that $\underset{(x,\rho)}{\max} -\frac{xb\chi''_1+d\chi'_1}{\chi_1}<M$ for some $M>0$ and let $c_{\min}=\underset{[0,1]\times [0,1]}{\min}c(x,\rho)$.
					We choose another two non-negative functions $h_0(\rho), h_\epsilon(\rho)$ so that
					\begin{gather*}
						h_0(0)=0, h'_0(0), h''_0(0)=0, h_0(\rho)=\frac{c_{min}}{2}-M(1-\rho)\  \text{on}\  [1-\frac{c_{min}}{2M},1] \\
						h_\epsilon\in C_c^2([0,1-\frac{c_{min}}{2M}+\epsilon]), h_\epsilon(0)=0, h'_\epsilon(0)=V'(0), h''_\epsilon(0)=V''(0),\\ L_\rho h_\epsilon<0\ \text{on}\ [0,1-\frac{c_{min}}{2M}].
					\end{gather*}
					Then for $\rho\in[1-\frac{c_{min}}{2M}, 1]$,
					\begin{gather}\label{A}
						\frac{(L_\rho+xe(\rho-1)\chi'_1\partial_\rho)h_0}{h_0}>-\frac{xb\chi''_1+d\chi'_1}{\chi_1}.
					\end{gather}
					Indeed we can rescale $\chi_1$ so that the coefficient of $\partial_\rho$ in $L_\rho+xe(\rho-1)\chi'_1\partial_\rho$ is bigger than $\frac{c_{min}}{2}$, so the left hand side term is no less than $M$, hence \eqref{A} is true.
					
					We let
					\begin{gather*}
						V(x,\rho)=c_0f(x,\rho)h_0(\rho)+h_\epsilon(\rho)\\
						\text{where}\ 
						f(x,\rho)=\chi_0(x)f_0(\rho)+\chi_1(x)f_1(\rho), c_0>0\ \text{is a constant}.
					\end{gather*}
					We want to ensure that $L(fh_0)<0$ on $\rho\in[1-\frac{c_{min}}{2M},1]$. The coefficient of $f_1$ in $L(fh_0)$ is \begin{gather}\label{coefficient}
						(xb\chi''_1+d\chi'_1)h_0+\chi_1(x)L_\rho h_0+xe(x,\rho)(\rho-1)\chi_1'(x)h'_0(\rho)
					\end{gather}
					through a direct computation in\eqref{lf} below, which is positive by \eqref{A}. So we can subtract a positive constant from $f_1$ so that $L(fh_0)<0$ when $\rho\in [1-\frac{c_{min}}{2M},1]$. We compute:
					\begin{align}\label{lf}
						Lf=&\chi_0(x)L\rho f_0+\chi_1(x)L_\rho f_1+xb(x,\rho)\left[\chi''_0 f_0+\chi''_1f_1\right]\\\nonumber
						&+d(x,y)\left[\chi'_0f_0+\chi'_1f_1\right]+x(\rho-1)e(x,y)\left[\chi'_0f'_0+\chi'_1f'_1\right]\\\nonumber
						L(f(x,\rho)h_0(\rho))=&(Lf)\cdot h_0+f\cdot L_\rho h_0+h'_0(\rho)\left[xe(x,\rho)f_x+2(xa(x,\rho)+\rho^2)f_\rho\right]
					\end{align}
					
					Finally for $c_0>0$ small enough, $V$ is positive since $h_\epsilon>0$ and $L_\rho h_\epsilon<0$ for $\rho\in [0, 1-\frac{c_{min}}{2M}+\epsilon]$. We adjust $h_\epsilon$ so that $\underset{\rho\in[1-\frac{c_{min}}{2M},1]}{\max}L_\rho h_\epsilon<\underset{\rho\in[1-\frac{c_{min}}{2M},1]}{\max} -L_\rho (c_0fh_0)$ and $LV<0$. Such $V$ is finite by construction, so $\frac{LV}{V}<0$. $V$ satisfies the boundary conditions at $\rho=0$ and when $x\in [1-\epsilon, 1]$, $V$ only depends on $\rho$, so this $V$ could be extended to a global function in this layer.

					By induction, we construct $V$ layer by layer, and such $V$ is positive on $P\setminus H$ while $V$ is $\infty$ at corners which are the intersection of two quadratic edges and on the quadratic edges where $\rho$ is constant. We have finitely many layers and each layer is covered by finite closed rectangles or irregular areas, and we showed that $\frac{LV}{V}<0$ in each of this area, so $V$
					is a Lyapunov function indeed satisfying $LV\leq\lambda_0V$ for some $\lambda_0<0$.
				\end{proof}
				\begin{remark}
					The Lyapunov function $V$ we constructed above behaves like $\rho^c$ where $\rho$ is the  distance to the tangent edge and $c$ is some positive constant. 
				\end{remark}
				
				$V$ gives rise to a norm $||\cdot||_{P_V}$:
				\begin{gather*}
					||f||_{P_V}:=\underset{P}{\text{sup}}\ \frac{f^+}{V}+\underset{P}{\text{sup}}\ \frac{f^-}{V}.
				\end{gather*}
				and we define the space $C^V(P)$ :
				\begin{gather*}
					C^V(P):=V\cdot C^0(P)
				\end{gather*}
				equipped with the norm $||\cdot||_{P_V}$. By saying $f\succ 0$ if $||f||_{P_V}>0$, let
				\begin{gather*}
					\lambda_{0}:=\underset{\lambda\in\mathbb{R}}{\inf}\ \{Af\leq\lambda f,\  0\prec f\}.
				\end{gather*}
				Clearly, $\lambda_0<0$. A similar argument of Proposition \ref{prop_lambda0} gives rise to the following result:
				\begin{proposition}\label{p_v}
					There exists some $M>0$ so that for $f\in C^V(P)$, $t>0$, 
					\[||P_tf||_{P_V}\leq M||f||_{P_V}e^{\lambda_0t}.\]
				\end{proposition}
				
				Fix a point $p\in P$ neither on $H$ nor any quadratic edge, we denote $p_t(p,\cdot)$ the transition probability of the diffusion starting from $p$. A corollary is 
				\begin{corollary}\label{coro_p}
					For $f\in C^V(P), t>0$
					\begin{gather}\label{f_pv}
						|P_tf(p)|\leq V(p)||P_tf||_{P_V}\leq MV(p)||f||_{P_V}e^{\lambda_0t}.
					\end{gather}
				\end{corollary}
				
				\begin{remark} \label{rem:topological}
					Consider the example in \cite{topological}: on the triangle $T=\{(x,y):0\leq x,y,x+y\leq 1\}$,
					\begin{gather*}
						L=\gamma_{12}\left[xy(\partial_x-\partial_y)^2+(y-x)(\partial_x-\partial_y)\right]\\
						+\gamma_{23}\left[y(x\partial_x+(y-1)\partial_y)^2+(y-1)(x\partial_x+(y-1)\partial_y)\right]\\
						+\gamma_{13}\left[x((x-1)\partial_x+y\partial_y)^2+(x-1)((x-1)\partial_x+y\partial_y)\right].
					\end{gather*}
					In this case $1-x-y=0$ is the only tangent edge. $V=1-x-y$, which measures the scaled distance between $(x,y)$ and the tangent edge, is a Lyapunov function since 
					\begin{align*}
						LV=-(1-x)\gamma_{13} + (1-y) \gamma_{23})V\leq -\min(\gamma_{13},\gamma_{23})V.
					\end{align*}
				\end{remark}
				
				\subsection{Exponential convergence to invariant measure}
				First, there exists a tubular neighborhood of $H$ that is diffeomorphic to $U=[0,1]\times [0,1]$ with $H$ mapped to $[0,1]\times\{0\}$. We fix a non-negative function $h\in C_c(U)$ satisfying $h|_H\equiv 1, h|_{y=1}\equiv 0$. Then we integrate the transition probability $p_t(p,x,y)$ with respect to $h$ in $U$:
				\begin{gather}\label{weighted_prob}
					q_h(t,x):=\int_0^1 p_t(p,x,y)h(y)dy.
				\end{gather}
				It can be interpreted as the marginal density with respect to coordinate along $H$ ($x$) and limited to some neighbourhood of $H$ (specified by $h$).
				
				We already know that there exists a unique invariant measure $\mu_0$ supported on $H$. The total variation distance between $q_h(t,x)$ and $\mu_0(x)$ can be bounded by $\chi^2$-divergence or Kullback-Leibler (KL) divergence. To show the exponential decay of the  $\chi^2$-divergence or the Kullback-Leibler(KL), convergence requires a {Poincar\'e} inequality or a logarithmic Sobolev inequality, respectively, for $\mu_0$. It turns out that if $L|_H$ has a Kimura endpoint, then the $\chi^2$-divergence is not finite. And if $L|_H$ has a quadratic endpoint, then $\mu_0$ does not satisfy the logarithmic Sobolev inequality. So in the proof of the following theorem, we employ an  interpolation divergence as constructed in \cite[Remark 3.2(3)]{10.1214/07-AIHP152}. 
				
				For $1<r\leq 2$, define
				\begin{gather}
					\psi(u)=3+\left(\frac{5}{2}-\frac{3}{r-1}\right)(u-3)+\frac{9}{r(r-1)}\left[\left(\frac{u}{3}\right)^r-1\right], u\geq 0.
				\end{gather}
				Such $\psi$ is convex, $\psi(1)=0$ and $\psi(u)\sim \frac{9}{r(r-1)3^r}u^r$ at $+\infty$.
				\begin{lemma}\cite[Lemma 1.1]{10.1214/07-AIHP152}\label{total_variation_bound}
					There exists a constant $C_\psi>$ such that for two probability measures $P, Q$, 
					\begin{gather*}
						||P-Q||_{TV}\leq C_\psi\sqrt{\int \psi\left(\frac{dP}{dQ}\right)dQ}.
					\end{gather*}
					If $\mu$ satisfies a Poincar\'e inequality, then for $h\geq 0$ satisfying $\int h\mu=1$, we have that $\mu$ satisfies the $I_\psi$-inequality:
					\begin{gather}
						\int\psi(h)\mu\leq C_\psi\int\psi''(h)|\partial_xh|^2\mu.
					\end{gather}
				\end{lemma}
				\begin{proposition}
					Assume for some $0<A<1$ that
					\begin{gather}\label{h(y)}
						h\geq 0, \ h|_{H\times [0,A]}\equiv  1\ \text{for}\ y\in [0, 1], \ h|_{y=1}, h'|_{y=1}\equiv 0.  
					\end{gather} 
					Fix a point $p$ that is not on any quadratic transverse edge, then the total variation between the $h$-weighted transition probability $q_h(t,x)$ defined in \eqref{weighted_prob} and the invariant measure $\mu_0$ has exponential decay:
					\begin{gather*}
						||q_h(t,x)-\mu_0||_{TV}\leq Me^{-\lambda t},
					\end{gather*}
					where $\lambda=\min\frac{1}{2}(\frac{1}{C_\psi}-\epsilon, -(1-\epsilon)\lambda_0)$  $\forall\epsilon>0$.
				\end{proposition}
				\begin{proof}
					Three are three types of boundary of $L|_H$. We first take a coordinate change at quadratic endpoints. By abuse of notation, we still denote $x$ the coordinate along $H$ and $U$ is mapped to $I\times [0,1]$.
					
					On $I\times [0,1]$, $L$ takes the form
					\begin{gather}
						L=a(x,y)\partial_{xx}+b(x,y)y^{m}\partial_{yy}+c(x,y)y\partial_{xy}+d(x,y)\partial_x+e(x,y)^{m-1}\partial_y,\ \  m\in\{1,2\}.
					\end{gather}
					The proofs for the cases $H$
					being a quadratic edge or Kimura edge are essentially the same. So in the following proof, we assume that $m=2$. 
					In addition, we will prove below that in an appropriate set of variables, we may choose
					\begin{align}\label{wlog:c=0}
						c(x,1)=0.
					\end{align}

					Regarding $q_h(t,x)$ as the marginal density limited to neighbourhood of $H$, its total measure approaches 1 at an exponential rate of $\lambda_0$. Indeed 
					\begin{align}\label{q(t,x)}
						\int_I q_h(t,x)dx&=\int_I\int_0^1p_t(p,x,y)h(y)dxdy=1-\int_I\int_0^1p_t(p,x,y)[1-h(y)]dxdy\\
						&\geq 1-Me^{\lambda_0 t}\nonumber
					\end{align} for some $M>0$. Moreover, it satisfies the backwards equation:
					\begin{gather}
						\partial_tq(t,x)=\int_0^1L^*p_t(p,x,y)h(y)dy=L_t^*q_h(t,x)+v(t,x)
					\end{gather}
					where by integrating by parts
					\begin{align*}
						L_t^*q(t,x)=&\partial_{xx}\int_0^1 a(x,y)p_t(p,x,y)h(y)dy-\partial_x\int_0^1c(x,y)yp_t(p,x,y)h'(y)dy\\
						&-\partial_x\int_0^1d(x,y)p_t(p,x,y)h(y)dy\\
						=:&\partial_{xx}(B(t,x)q_h(t,x))-\partial_x(A(t,x)q_h(t,x))
					\end{align*}
					\begin{gather}\label{def:v(t,x)}
						v(t,x)=\int_0^1 b(x,y)y^2p_t(p,x,y)h''(y)dy+\int_0^1e(x,y)yp_t(p,x,y))h'(y)dy.
					\end{gather}
					We make a few remarks on these terms.
					First, $||v(t,x)||_1$ and  $||[A(t,x)-d(x,0)]q_h(t,x)||_1$ decay exponentially at the rate $\lambda_0$.
					Second, $c(x,1)=0$ ensures that $\frac{c(x,y)yh'(y)}{h(y)}$ is bounded, so that by choosing $U$ sufficient small we may assume that
					\begin{gather}\label{A(t,x)}
						\big|\partial_x[B(t,x)-a(x,0)]\big|+\big|\frac{\mu'_0}{\mu_0}(B(t,x)-a(x,0))\big|+\big|(A(t,x)-d(x,0))\big|\leq\epsilon.
					\end{gather}

					We are now ready to estimate $||q_h(t,x)-\mu_0(x)||_{TV}$. Let $f(t,x)=\frac{q_h(t,x)}{\mu_0(x)}$.
					Then by Lemma \ref{total_variation_bound}, 
					\begin{gather}\label{total_variance}
						||q_h(t,x)-\mu_0(x)||_{TV}\leq \frac{C_\psi}{P(t)}\sqrt{\int \psi(f(t,x))\mu}+Me^{-\lambda_0 t},
					\end{gather}
					where $\mathbb{P}(t)=||q_h(t,x)||_1$.
					We differentiate the term
					\begin{align*}
						\frac{d}{dt}(\psi(f(t,x)),\mu_0)=&(\psi'(f),\partial_tq_h(t,x))\\
						=&(L_0\psi'(f),f\mu_0)+((L_t-L_0)\psi'(f),f\mu_0)+(\psi'(f),v(t,x))\\
						=:&-\int_I\psi''(f)(\partial_xf)^2\mu_0+II+III.
					\end{align*}
					We will show below that 
					\begin{gather}
						\label{TVthm_II}
						II\leq 2\epsilon\int_I\psi''(f)(\partial_xf)^2\mu_0+\epsilon\int_I\psi(f)\mu_0dx+Me^{\lambda_0t}\\
						\label{TVthm_III}
						III\leq\epsilon\int_I\psi(f)\mu_0dx+Me^{(1-\epsilon)\lambda_0t}.
					\end{gather}
					Therefore, we have
					\begin{gather*}
						\frac{d}{dt}(\psi(f(t,x)),\mu_0)\leq -(1-3\epsilon)\int_I\psi''(f)(\partial_xf)^2\mu_0+2\epsilon\int_I\psi(f)\mu_0+Me^{\lambda_0t}\\
						\leq -\left[\frac{1-3\epsilon}{C_\psi}-2\epsilon\right]\int_I\psi(f)\mu_0+Me^{\lambda_0t}.
					\end{gather*}
					As a consequence,
					\begin{gather}\label{pinsker_inq}
						\int_I\psi(f(t,x))\mu_0\leq Me^{-at}
					\end{gather}
					where $a=\min(\frac{1-2\epsilon}{C_\psi}-2\epsilon, -\lambda_0)$. Finally combined with \eqref{total_variance}, we obtain that
					\begin{gather*}
						||q_h(t,x)-\mu_0(x)||_{TV}\leq Me^{-\frac{a}{2}t},
					\end{gather*}
					for some constant $M>0$.
					This only leaves the proof of \eqref{TVthm_II}, \eqref{TVthm_III}, and \eqref{wlog:c=0}. We leave the proof of \eqref{TVthm_III} to Lemma \ref{pointwise_bound} in the Appendix.
					\paragraph{Proof of \eqref{TVthm_II}} To see \eqref{TVthm_II}, we integrate by parts to see
					\begin{gather*}
						II=-(\psi''(f)(\partial_xf)^2, (B(t,x)-a(x,0))\mu_0)-(C(t,x)\partial_x\psi'(f), f\mu_0)=:II_1+II_2
					\end{gather*}
					where
					\[C(t,x)=\partial_x[B(t,x)-a(x,0)]+\frac{\mu'_0}{\mu_0}(B(t,x)-a(x,0))-(A(t,x)-d(x,0)).\]
					It is easy to see that 
					\begin{gather}\label{II}
						II_1\leq\epsilon\int_I\psi''(f)(\partial_xf)^2\mu_0
					\end{gather}
					since by \eqref{A(t,x)}
					$|B(t,x)-a(x,0)|\leq\epsilon$. We use H\"older and Cauchy-Schwarz inequalities to bound
					\begin{gather*}
						|II_2|\leq(C(t,x)(\psi''(f))^2(\partial_xf)^2|f|^{2-p},\mu_0)+(C(t,x)|f|^p,\mu_0)=:II_{21}+II_{22}.
					\end{gather*}
					Since $\psi''(f)|f|^{2-p}$ is bounded by some constant independent of $\epsilon$,  by rescaling $\epsilon'=\frac{\epsilon}{F}$ in \eqref{A(t,x)} we obtain
					\begin{gather}
						II_{21}\leq\epsilon\int_I\psi''(f)(\partial_x f)^2\mu_0.
					\end{gather}
					We split $II_{22}$ into two parts
					\begin{gather*}
						\int_{f\leq 1} C(t,x)|f|^p\mu_0+\int_{f\geq 1} C(t,x)|f|^p\mu_0.
					\end{gather*}
					The first term decreases exponentially as $e^{\lambda_0t}$ using $|f|^p\leq f$ and that $(C(t,x),q_h(t,x))$ decreases exponentially. The second term is bounded by $\int\psi(f)\mu_0$ multiplied by some constant. To see this, when $0\leq f\leq 3$, because of the convexity of $\psi$,
					\[\psi(f)-\psi(1)-\psi'(1)(f-1)\geq\frac{1}{2}\underset{0\leq f\leq 3}{\min}\psi''(f)(f-1)^2,\]
					and $\psi(f)\sim f^p$ near $\infty$. Thus,
					\begin{gather*}
						(f-1)^2\leq C(1+f^{2-p})(\psi(f)-\psi(1)-\psi'(1)(f-1)).
					\end{gather*}
					Hence 
					\begin{gather}
						\int_I\psi(f)\mu_0\geq \int\frac{(f-1)^2}{1+f^{2-p}}\mu_0\geq\int_{f\geq 1}\frac{(f-1)^2}{1+f^{2-p}}\mu_0\geq M\int_{f\geq 1}\psi(f)\mu_0.
					\end{gather}
					Gathering the above estimates,
					\begin{gather}
						II\leq II_1+II_{21}+II_{22}\leq\epsilon\int_I\psi''(f)(\partial_xf)^2\mu_0+\epsilon\int_I \psi(f)\mu_0+Me^{\lambda_0t}.
					\end{gather}
					\paragraph{Proof of \eqref{wlog:c=0}}
					If we consider the new coordinate $(t,y)$, then the coefficient of mixed derivative $\partial_{tx}$ is
					\[A(x,y)=cyt_x+2by^2t_y\]
					To make $A(x,1)=0$, we choose \[t_x(x,1)=1, \qquad \mbox{ and then} \qquad t_y(x,1)=-\frac{c(x,1)}{2b(x,1)}.\] 
					Let $f(x,y)=\partial_{xy}t=\partial_{yx}t$, we choose $f(x,y)=f(x,1)h(y)$ and
					\begin{gather}\label{t}
						t_x(x,y)=\int_0^yf(x,z)dz+g(x),\ 
						t_y(x,y)=\int_0^xf(z,y)dz+l(y).
					\end{gather}
					Plugging (\ref{t}) into $t_x, t_y$, we find
					\begin{gather*}
						\int_0^1f(x,y)dy+g(x)=1,\ 
						\int_0^xf(z,1)dz+l(1)=-\frac{c(x,1)}{2b(x,1)},
					\end{gather*}
					and hence
					\[f(x,1)=-\frac{\partial }{\partial x}\frac{c(x,1)}{2b(x,1)},\ g(x)=1-f(x,1)\int_0^1h(y)dy.\]
					Now we only need to make $t_x$ everywhere positive. Then, the coordinate change $(x,y)\mapsto (t,y)$ would be bijective on $[0,1]^2$. Now by (\ref{t}),
					\begin{gather*}
						t_x(x,y)=f(x,1)\int_0^yh(z)dz+1-f(x,1)\int_0^1h(y)dy
						=1-f(x,1)\int_y^1h(z)dz.
					\end{gather*}
					We can choose a smooth function $h(z)$ so that $|f(x,1)|\cdot|\int_y^1h(z)dz|<1$ for all $(x,y)\in [0,1]^2$. Therefore $t(x,y)$ exists since $t_{xy}=t_{yx}$ and $A(x,1)=0$.
				\end{proof}
				As a consequence we have the following estimate in Wasserstein distance,
				\begin{theorem}\label{col:wasserstein}
					Fix a point $p$ that is not on any quadratic transverse edge. Then the Wasserstein distance between the transition probability $p_t(p,\cdot)$ and the invariant measure supported on the tangent edge converges exponentially:
					\begin{align}
						W(p_t(p,\cdot),\mu_0(x)\delta_0(y))\leq M e^{-\frac{a}{2} t}, t>0.
					\end{align}
				\end{theorem}
				\begin{proof}
					For any $f\in C^0(P)$ with $\text{Lip}(f)\leq 1$, 
					\begin{gather*}
						\int_Pfp_t-\int_Pf\mu(x)\delta_0(y)=\int_{P\setminus U}fp_t+\int_Ufp_t-\int_Uf\mu(x)\delta_0(y).
					\end{gather*}
					$\int_{P\setminus U}fp_t$ is bounded by $Me^{-\lambda_0t}$with constant $M$ independent of $f$. The second term is
					{
					\begin{align*}
						&\int_Ufp_t(p,x,y)dxdy-\int_Uf\mu(x)\delta_0(y)\\
						=&\int_I\int_0^1 p_t(p,x,y)[f(x,y)-f(x,0)h(y)]dxdy\\&+\left[\int_I\int_0^1 p_t(p,x,y)f(x,0)h(y)dxdy-\int_If(x,0)\mu(x)dx\right].
					\end{align*}
				}
					It is bounded by $||q_h(t,x)-\mu(x)||_{TV}$. The first term also has exponential decay because $[f(x,y)-f(x,0)h(y)]\in C^V(P)$ with bounded norm independent of $f$. Hence the Wasserstein distance between between $p_t(p,\cdot)$ and $\mu(x)\delta_0(y)$ decays exponentially  with the same rate $\frac{a}{2}$.				\end{proof}

				\section{Numerical Experiments}\label{sec:num}
				In this section, we present numerical simulations that illustrates the long time behaviors obtained in earlier sections.  Since the distributions may converge to lower dimensional  manifolds at tangent boundaries, we adopt the probabilistic representation \cite{stroock1997multidimensional} that considers the stochastic process governed by the generator $L$. Comparing with solving PDE directly on meshgrids, the probabilistic representation can be numerically implemented in a way respecting the local coordinate chart at degenerated boundaries. See \cite{wangKPP,ferre2019error} for examples of stability of long time behaviour simulation by numerical probabilistic approach. 
				
				Following the theoretical derivations, we start with examples in one space dimension before considering a two dimensional case. 
				
				We present the Lagrangian scheme in detail in the one-dimensional setting.
				\subsection{One dimensional cases with mixed boundary conditions}
				We first consider a general model operator that is of Kimura type at $x=0$ and of quadratic type at $x=1$.  Specifically,
				\begin{align}
					L=(c_0(1-x)^2-c_1x(1-x))\partial_x + x(1-x)^2\partial_x^2,
				\end{align}
				where $c_0$ and $c_1$ denote the strength of convection near boundaries. We recall that a Kimura boundary at $x=0$ is transverse  when $c_0>0$ and tangent when $c_0=0$; the quadratic boundary at $x=0$ is transverse when $c_1>1$ and tangent when $c_1<1$. When both boundaries are transverse, $L$ admits an invariant measure
				\begin{align}
					\mu_0=\frac{1}{Z}x^{c_0-1}(1-x)^{c_1-2}.
				\end{align}
				To present the convergence result, we use particles to represent the distribution $q(t,x)$ defined by
				\begin{align}
					q(t,x)=\int_0^1 p_t(x,y)q_0(x) dy,
				\end{align}
				where $q_0$ is the initial distribution and $p_t$ denotes the transition probability of the process whose generator is $L$. We denote the process as $X_\cdot$, where $X_0$ has initial distribution $q_0$, and $X_\cdot$ follows stochastic differential equation (SDE) in the It\^o sense \cite{stroock1997multidimensional},
				\begin{align}
					dX=(c_0(1-X)^2-c_1X(1-X))dt+\sqrt{2X(1-X)^2} dW_t,
				\end{align}
				where $W_t$ is a standard Brownian motion. The probability density function of $X_t$ is then $q(t,\cdot)$.
				
				We now develop a modified Euler Maruyama (EM) scheme \cite{kloeden1992stochastic} to generate approximated realizations of the  process $X$ and then approximate $q$ by the empirical distribution of such generated realizations. First, $X_n$ denotes the approximated process at time $t_n=n\Delta t$. Here, the length of the time interval  $\Delta t$ is a fixed constant during the computation. Since the degeneracy of $L$ only happens at boundaries, we apply a standard EM scheme  when $x$ is `far' from the boundary. 
				
				In practice, when $X_n\in[0.01,0.99]$, starting from $X_n$, for one time step $\Delta t$, we consider,
				\begin{align}
					X_{n+1}=X_n+(c_0(1-X_n)^2-c_1X_n(1-X_n))\Delta t+\sqrt{2X_n(1-X_n)^2}\sqrt{\Delta t} \mathcal{N},
				\end{align}
				where $\mathcal{N}$ denotes a generated standard Gaussian distribution. 
				
				When $X_n<0.01$, the particle is considered as approaching the Kimura boundary at $x=0$. The leading order of SDE turns to be
				\begin{align*}
					dX=c_0dt,
				\end{align*}
				which provides a strong forcing so that the position of particles remain strictly positive. To this end, we consider an implicit scheme,
				\begin{align*}   
					X_{n+1}=X_n+(c_0(1-X_n)^2-c_1X_{n}(1-X_n))\Delta t+\sqrt{2X_{n+1}(1-X_n)^2}\sqrt{\Delta t} \mathcal{N},
				\end{align*}
				which may be explicitly solved by
				{\scriptsize
				\begin{align}
					\sqrt{X_{n+1}}=\frac{\sqrt{2\Delta t}\Delta W(1-X_n)+\sqrt{(\sqrt{2\Delta t}\Delta W(1-X_n))+4(X_n+(c_0(1-X_n)^2-c_1X_{n}(1-X_n))\Delta t)}}{2}.
				\end{align}
			}
				When $X_n>0.99$, the particle is considered as approaching the quadratic boundary at $x=1$. The leading order of the SDE is
				\begin{align*}
					dX=-c_1(1-X)dt+\sqrt{2} (1-X) dW_t.
				\end{align*}
				It is then natural to apply a logarithm transform over $1-X$, i.e., let $y=-\log(1-x)$. Then by the It\^o formula, $Y$ follows
				\begin{align}
					dY=(c_0(1-X)-c_1X)dt+\sqrt{2X} dW_t + X dt,
				\end{align}
				where $X=1-\exp(-Y)$. This is not a degenerate SDE since here $X$ is assumed to be near $x=1$.
				
				\paragraph{Case I: Transverse Kimura and  Transverse Quadratic}
				In Fig.\ref{fig:QtrKtr}, we show the empirical PDF (histogram) of the approximated realizations obtained by the proposed algorithm. We start from a single-point initial distribution, $q_0=\delta(x-0.5)$ with one million realizations. To ensure transverse boundary condition on both side, we set $(c_0,c_1)=(0.5,2)$. For each realization, we approximate the SDE with $\Delta t=2^{-14}$ until $T=2^{1}$. As a reference, we plot the target invariant measure $\mu_0$ with a solid red line. 
				
				\begin{figure}[htbp]
					\centering
					\includegraphics[width=\linewidth]{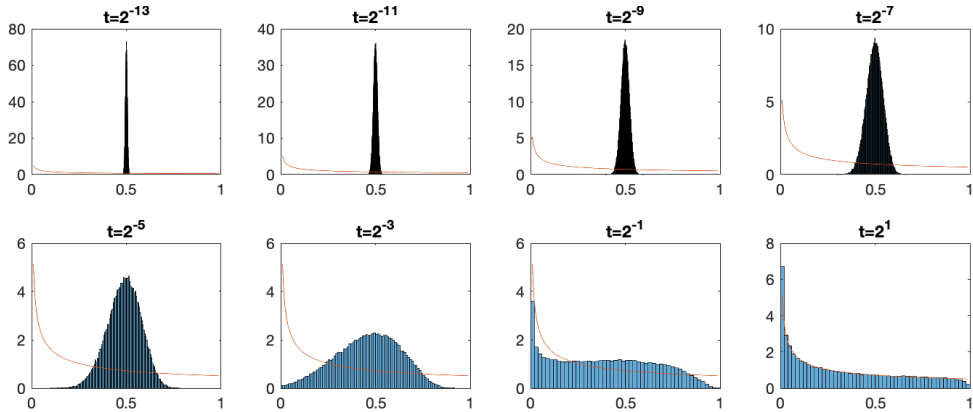}
					\caption{Distribution of realizations when $x=0$ is of Kimura transverse type and $x=1$ is of quadratic transverse type. $\mu_0$ is the solid red line in each sub-graph}
					\label{fig:QtrKtr}
				\end{figure}
				In addition, we compute the Wasserstein distance between to empirical distribution of the simulation and the target invariant measure. By optimal transport theory in one dimension, the optimal transport map in the Wasserstein distance is given by the inverse cumulative distribution functions \cite{rachev1998mass}. More precisely, let $\{x_i\}_{i=1}^N\subset R^1$ be the position of the particles and without loss of generality, assume $x_1\leq x_2\leq\cdots\leq x_N$. Let $F_\mu$ be the cumulative distribution functions of target distribution. Then the $W^p$ distance can be approximated by
				\begin{align}
					W^p=\Big(\frac{1}{N}\sum_{i=1}^{N}\big(x_i-F_\mu^{-1}(\frac{i-1/2}{N})\big)^p\Big)^{1/p}
				\end{align} 
				In Fig.\ref{fig:WD1D}(a), we present the $W^1$ distance between the empirical distribution of realization of SDE with $\Delta t=2^{-9}$ until $T=2^{1}$.
				\paragraph{Case II: Transverse Kimura and Tangent Quadratic}
				In the second case, we investigate the case when one of the boundary conditions becomes tangent. As an example, we set $(c_0,c_1)=(0.5,0.5)$. Then the quadratic boundary on the right becomes tangent and the invariant measure is $\delta(x-1)$. The $p$-Wasserstein distance in such case can be computed by,
				\begin{align}
					W^p(q(t,\cdot),\delta_1)=\Big(\int (1-x)^p q(t,x)dx\Big)^{1/p}.
				\end{align}
				In Fig.\ref{fig:WD1D}(b), we present the $W^1$ distance between the empirical distribution of realizations of the SDE with the invariant measure.
				\begin{figure}[htbp]
					\centering
					\begin{subfigure}{0.35\linewidth}
						\includegraphics[width=\linewidth]{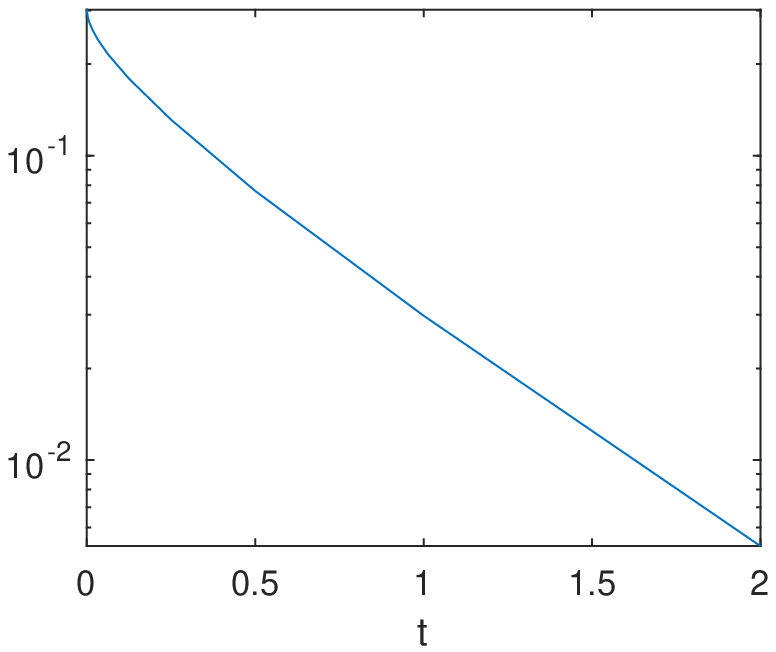}
						\caption{ $x=0$ is of Kimura transverse type and $x=1$ is of quadratic transverse type.}
						\label{fig:QtrKtr2}
					\end{subfigure}
					~
					\begin{subfigure}{0.35\linewidth}
						\includegraphics[width=\linewidth]{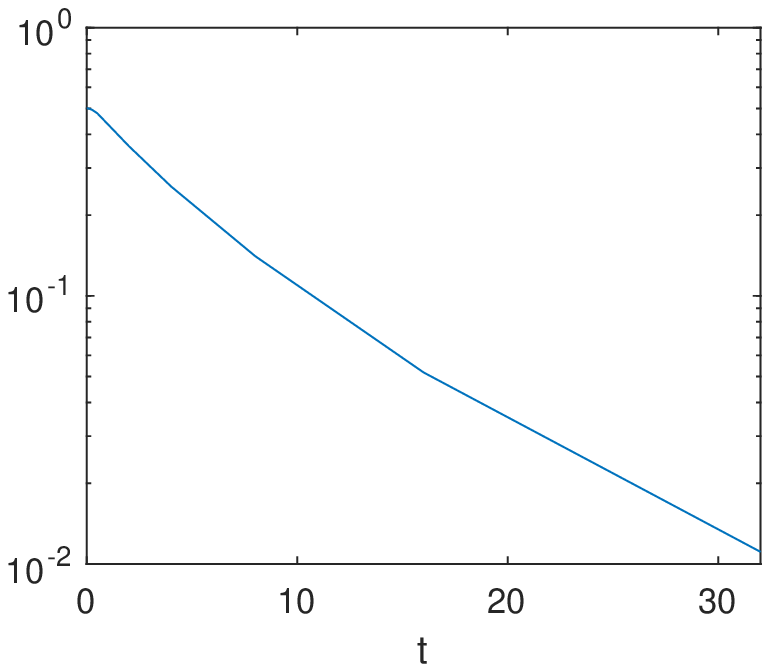}
						\caption{ $x=0$ is of Kimura transverse type and $x=1$ is of quadratic tangent type}
						\label{fig:QtKtr}
					\end{subfigure}
					\caption{$W^1$ distance between the empirical distribution and invariant measure}\label{fig:WD1D}
				\end{figure}
				
				\subsection{2D with tangent diagonal}
				We now turn to the long time behavior for a degenerate operator in two dimension. Specifically we consider, $L$ defined in a triangular domain $T=\{x\geq0,y\geq0,x+y\leq1\}$, 
				\begin{align}\label{eqn:2doperator}
					\mathcal{L}=&\gamma_{12}\left[x y \varphi_{3} \otimes \varphi_{3}: \nabla^{2}+\left(y-x\right) \varphi_{3} \cdot \nabla\right]+\gamma_{23}\left[y \varphi_{2} \otimes \varphi_{2}: \nabla^{2}+\left(y-1\right) \varphi_{2} \cdot \nabla\right]\\\nonumber
					&+\gamma_{13}\left[x \varphi_{1} \otimes \varphi_{1}: \nabla^{2}+\left(x-1\right) \varphi_{1} \cdot \nabla\right]\nonumber
				\end{align}
				where,
				\begin{align} \nonumber
					\varphi_{3}=\left(\begin{array}{c}1 \\ -1\end{array}\right), \quad \varphi_{2}=\left(\begin{array}{c}x \\ y-1\end{array}\right), \quad \varphi_{1}=\left(\begin{array}{c}x-1 \\ y\end{array}\right).
				\end{align}
				Such operator in Eq.\eqref{eqn:2doperator} is the asymptotic generator of two reflection coefficients in a stochastic homogenization regime \cite{topological}. 
				Then we construct an SDE $X=\left(\begin{array}{c}x \\ y\end{array}\right)$ which follows,
				\begin{align}\label{eqn:2dsde}
					dX=&(\gamma_{12}(y-x)\varphi_3+\gamma_{23}(y-1)\varphi_2+\gamma_{13}(x-1)\varphi_1)dt\\\nonumber
					&+\sqrt{2\gamma_{13}x}\varphi_1dW_1+\sqrt{2\gamma_{23}y}\varphi_2dW_2+\sqrt{2\gamma_{12}xy}\varphi_3dW_3,
				\end{align}
				where $W_i$ for $i=1,2,3$ are independent standard Brownian motions and $X(0)\in T$. 
				We observe that $x=0$ and $y=0$ are transverse Kimura boundaries while $x+y=1$ is tangent quadratic boundary.
				As in Theorem \ref{col:wasserstein}, the invariant measure can be represented by $ \mu(x)\delta_{x+y=1}$.  $\mu$ may be computed by restricting $L$ on $x+y=1$. More precisely,
				\begin{align}
					L|_{x+y=1}=& (\gamma_{12}(1-2x)+\gamma_{23}x^2+\gamma_{13}(1-x)^2)\partial_x\\
					&+(1-x)x(\gamma_{12}+\gamma_{23}x+\gamma_{13}(1-x))\partial_x^2,
				\end{align}
				and $L$ admits an invariant measure density (along the diagonal edge),
				\begin{align}\label{eq:rho}
					\mu=\frac{(\gamma_{12}+\gamma_{13})(\gamma_{12}+\gamma_{23})}{(\gamma_{12}+\gamma_{13}(1-x)+\gamma_{23}x)^2},\quad x\in[0,1].
				\end{align}
				
				Starting with $X(0)=(0.1,0.1)^T$, we numerically integrate the SDE \eqref{eqn:2dsde} with $\gamma_{12}=1$, $\gamma_{13}=2$ and $\gamma_{23}=1$  by a Euler method with $\Delta t=2^{-9}$ until $T=2^2$ with $10^6$ realizations. To address the degeneracies near the boundary, we apply a direct cutoff at $x=0$ and $y=0$ at the transverse Kimura boundaries. The is due to the convection at these boundaries points to interior of domain and does not degenerate, comparing with quadratic cases. The process does not accumulate on these boundaries.  For boundary $x+y=1$, we first observe that the convection of continuous dynamics are in the direction $(1,-1)^T$. So at the $n$th-step of the Euler iteration starting from $(x_{n-1},y_{n-1})^T$, if $x_n+y_n>1$, we apply a re-scale factor $\frac{x_{n-1}+y_{n-1}}{x_{n}+y_{n}}$ to $x_n$ and $y_n$ to ensure the dynamics only propagate in the $(1,-1)^T$ direction.
				
				In Fig.\ref{fig:2dhist2}, we present the 2D histogram of the joint distribution of $(x,y)$ for different times. This shows that the majority of particles first travel rightwards along the $x$ axis before converging to the invariant distribution along the diagonal $x+y=1$, as predicted theoretically. 
				\begin{figure}[ht!]
					\centering
					\includegraphics[width=\linewidth]{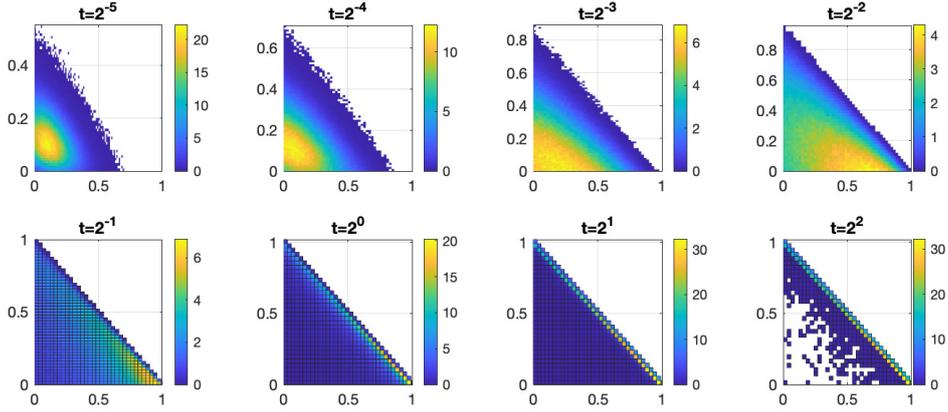}
					\caption{2D Histogram of the joint distribution of $(x,y)$ at different times}
					\label{fig:2dhist2}
				\end{figure}
				In Fig.\ref{fig:2dhist}, we present the histogram of the marginal distribution of $x$. It is found to converges to the invariant measure on the line $x+y=1$ given in \eqref{eq:rho}.
				\begin{figure}[ht!]
					\centering
					\includegraphics[width=0.9\linewidth]{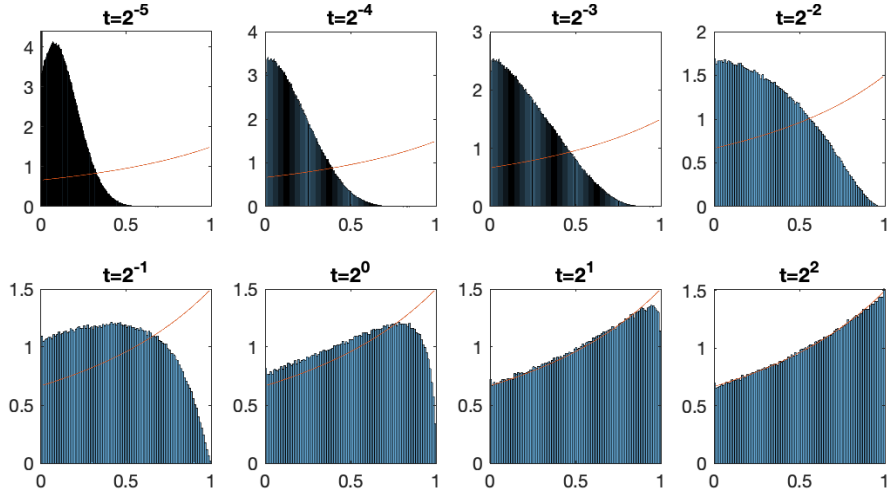}
					\caption{Histogram of the marginal distribution of $x$ at different times. Reference pdf is the solid red line.}
					\label{fig:2dhist}
				\end{figure}
				
				In Fig.\ref{fig:2dw1}(a) we show the approximate $E(1-X-Y)=\int (1-x-y)q(t,x,y)dxdy$ against time $t$. This confirms that the density of dynamics converges exponentially fast to the diagonal $x+y=1$.
				To quantitatively describe the convergence to the invariant measure, we also computed the Wasserstein distance between $x$ and $\mu(x)$ in Fig.\ref{fig:2dw1}(b). 
				\begin{figure}[ht]
					\centering
					\begin{subfigure}{0.35\linewidth}
						\includegraphics[width=\linewidth]{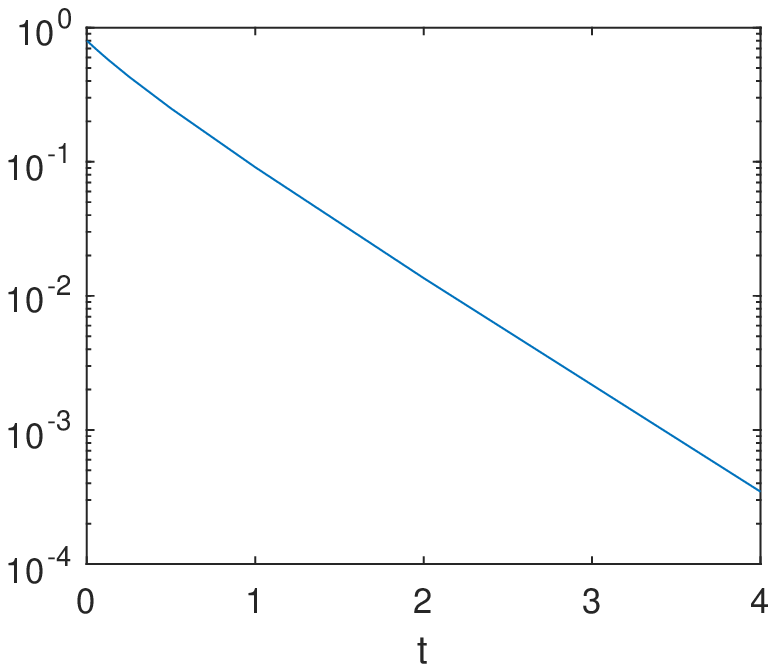}
						\caption{$E(1-x-y)$ in 2D triangular example.}
					\end{subfigure}~
					\begin{subfigure}{0.35\linewidth}
						\includegraphics[width=\linewidth]{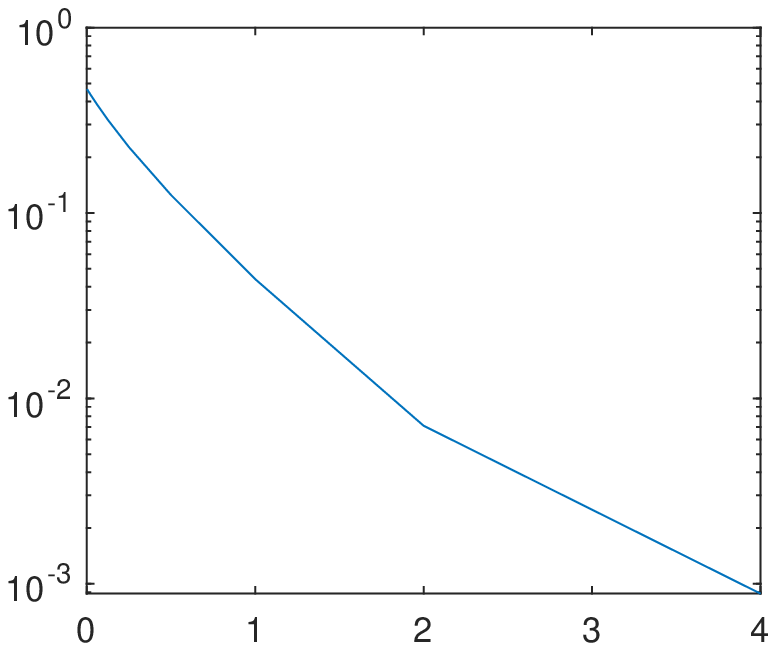}
						\caption{$W^1$ distance between empirical $x$ and $\rho(x)$ in 2D triangular example.}
					\end{subfigure}
					\caption{Convergence to degenerated invariant measure}
					\label{fig:2dw1}
				\end{figure}

				\newpage

				\appendix
				\section{Appendix} 
				\begin{lemma}\label{pointwise_bound}Fix $p\in P$ that is not on the quadratic edge or tangent Kimura edge.
					For $q\in H\times [A,1]$,
					the transition probability $p_t(p,q)$ has a pointwise upper bound:
					\begin{gather}\label{pointwise_esti}
						p_t(p,q)\leq C\cdot\text{exp}\left[-\frac{x^2}{4C_2t}+C_0|x|+C_1t\right],
					\end{gather}
					for some constant $C, C_0, C_1, C_2>0$.
					Therefore
					\begin{gather}\label{int_v(t,x)}
						\left|\int_I\psi'(f)v(t,x)dx\right|\leq\epsilon\int_I\psi\mu+\sqrt{t}e^{(1-\epsilon)\lambda_0 t}.
					\end{gather}
				\end{lemma}
				\begin{proof}
					We can follow the proof of \cite[Theorem 5.2.10]{saloff-coste_2001} with several modifications. 
					
					Firstly, unlike the setting under the usual Rimmannian measure on $P$, our proof introduces a weighted measure $d\mu$ on $P$. The principal symbol of $L$ induces a Riemannian metric $dV$ on $P$. For each Kimura boundary surface $H_i, 1\leq i\leq\eta$, define $\rho_i(p)$ to be the Riemannian distance from the point $p\in P$ to $H_i$. Then
					\[B_i|_{H_i}:=\frac{1}{4}L\rho_i^2|_{H_i}\]
					are coordinate-invariant quantities. We also let $B_i$ denote a smooth extension from $H_i$ to $P$ of the coefficients. Then we define the weighted measure $d\mu$ by
					\begin{gather}\nonumber
						d\mu(p):=\underset{i=1}{\overset{\eta}{\Pi}}\rho_i(p)^{2B_i-1}dV.
					\end{gather}
					In an adapted system of local coordinates of a corner $p$, $L$
					takes the form: $m+n=2$,
					\begin{align}
						L=&\underset{i=1}{\overset{m}{\sum}}x_i\partial_{x_ix_i}+\underset{i,j=1}{\overset{m,n}{\sum}}b(x,y)x_iy_j\partial_{x_iy_j}+\underset{j=1}{\overset{n}{\sum}}c(x,y)y_j^2\partial_{y_jy_j}\label{local_normal_coordinate} \\&+\underset{i=1}{\overset{m}{\sum}}d(x,y)\partial_{x_i}+\underset{j=1}{\overset{n}{\sum}}e(x,y)y_j\partial_{y_j}.\nonumber
					\end{align}
					The weighted measure $d\mu$ is a multiple by a smooth function of
					\begin{gather*}
						d\mu(x,y)=\underset{i=1}{\overset{m}{\Pi}}x_i^{d_i(x,y)-1}dx_i\underset{j=1}{\overset{n}{\Pi}}\frac{1}{y_j}dy_j.
					\end{gather*}
					
					For $\alpha\in\mathbb{R}$ and $\phi\in C^\infty(P)$ satisfying $|\nabla\phi|\leq 1$ under the Riemannian metric on $P$, define $H_t^{\alpha,\phi}(P)$ on $L^2(\mu)$ by
					\begin{gather*}
						H_t^{\alpha,\phi}f(x)=e^{-\alpha\phi(x)}H_t[e^{\alpha\phi}f](x).
					\end{gather*}
					Instead of the estimates of $||H_t^{\alpha,\phi}||_{2\to 2}$ in \cite{saloff-coste_2001}, we have
					\begin{gather}\label{2-estimate}
						||H_t^{\alpha,\phi}||_{2\to 2}\leq e^{(1+\epsilon)\alpha^2t+C_0|\alpha|t+C_1t}
					\end{gather}
					for some constants $C_0, C_1\geq 0$. 
					
					\paragraph{Proof of \eqref{2-estimate}} To see this, compute the derivative:
					\begin{gather*}
						\frac{\partial}{\partial_t}||H_t^{\alpha,\phi}f||_2^2=2(H_t^{\alpha,\phi}f,e^{-\alpha\phi}Le^{\alpha,\phi}H_t^{\alpha,\phi}f)_{L^2(\mu)}.
					\end{gather*}
					Taking $g=H_t^{\alpha,\phi}f$, to derive \eqref{2-estimate} it is sufficient to have
					\begin{gather}\label{sufficient}
						(e^{-\alpha\phi}g, Le^{\alpha\phi}g)\leq[\alpha^2+C_0|\alpha|+C_1](g,g).
					\end{gather}
					
					We associate a bilinear form $Q(u,v)$ to $-(Lu,v)_\mu$ by letting
					\begin{gather*}
						Q(u,v)=(Lu,v)_{L^2(\mu)},
					\end{gather*}
					which can be written as 
					\begin{gather*}
						Q(u,v)=Q_{\text{sym}}+(Tu,v)_{L^2(d\mu)}
					\end{gather*}
					where $Q_{\text{sym}}(u, v)$ is a symmetric bilinear form and $T$ is a first order vector field on $P$.
					We need to estimate it in two cases. 
					\paragraph{Case I:}In a neighborhood that is away from any Kimura edge, then $L$ is uniformly elliptic under $d\mu$ and of the form
					\begin{gather*}
						L=a\partial_x^2+2b\partial_{xy}+c\partial_y^2+V,
					\end{gather*}
					where $V$ is a vector field. Then
					\begin{gather*}
						Q_{\text{sym}}(u,v)=-\int [au_xv_x+cu_yv_y+b(u_xv_y+u_yv_x)]d\mu.
					\end{gather*}
					$T$ is a vector field with bounded smooth coefficients.
					Take $u=e^{\alpha\phi}g, v=e^{-\alpha\phi}g$. Then
					\begin{gather*}
						Q(u,v)=Q(g,g)+\alpha^2\int (a\phi^2_x+2b\phi_x\phi_y+c\phi^2_y)g^2d\mu+\alpha(T\phi,g^2)_{L^2(d\mu)}.
					\end{gather*}
					$L$ is uniformly parabolic so we can bound $||Q(g,g)||$ by $C||g||^2_{L^2(d\mu)}$ for some constant $C>0$. Using the fact that $|\nabla\phi|\leq 1$, the second term and the third term are bounded by $\alpha^2||g||^2_{L^2(d\mu)}$, $C|\alpha|\cdot||g||^2_{L^2(d\mu)}$ for some constant $C>0$, respectively. So we obtained \eqref{sufficient}.

					\paragraph{Case II:}In a neighborhood of the tangent Kimura edge, say suppose in a neighborhood of a corner which is the intersection of two Kimura edges, $L$ is degenerate. Under local coordinates \eqref{local_normal_coordinate},
					\begin{gather}\nonumber
						L=x\partial_{xx}+y\partial_{yy}+xyb(x,y)\partial_{xy}+d_1(x,y)\partial_x+d_2(x,y)\partial_y.
					\end{gather}
					We refer to \cite[§3.1]{article} to see that
					\begin{gather*}
						Q_{\text{sym}}(u,v)=-\int\left(xu_xv_x+yu_yv_y+\frac{1}{2}xyb(u_xv_y+u_yv_x)\right)d\mu
					\end{gather*}
					and $T$ is a tangent vector field with logarithmic singularities:
					\begin{gather*}
						T=x\left(\alpha_1+\beta_1\text{ln}x+\gamma_1y\text{ln}y\right)\partial_x+y\left(\alpha_2+\beta_2\text{ln}x+\gamma_2\text{ln}y\right)\partial_y
					\end{gather*}
					where the functions $\alpha_i, \beta_i, \gamma_i$ are smooth. Because $y=0$ is a tangent edge, the coefficient of $\text{ln}y\partial_x$ vanishes at the edge $y=0$.
					
					Take $u=e^{\alpha\phi}g, v=e^{-\alpha\phi}g$. Then,
					\begin{gather*}
						(Lu,v)=Q(g,g)+\alpha^2\int g^2(x\phi^2_x+y\phi^2_y+axy\phi_x\phi_y)d\mu+\alpha(T\phi,g^2)_{L^2(d\mu)}.
					\end{gather*}
					By \cite[Lemma 3.1]{article}, $||Q(g,g)||\leq C||g||^2_{L^2(d\mu)}$ for some constant $C>0$. Using the fact that $|\nabla\phi|\leq 1$, the second term is bounded by $\alpha^2||g||^2_{L^2(d\mu)}$. $T$ has logarithmic singularities near $x=0$ and we choose $\phi$ so that $|\phi_y\text{ln}x|\leq x$. Then the third term is also bounded by $C|\alpha|\cdot||g||^2_{L^2(d\mu)}$ for some constant $C>0$. We thus obtained \eqref{sufficient}.
					
					\paragraph{Proof of \eqref{pointwise_esti}}
					Fix $p\in P$ that is not on a quadratic edge or tangent Kimura edge, $q\in H\times [A,1], r_1, r_2>0$
					so that $B_{r_1}(p), B_{r_2}(q)\subset P$. Notice that if $q$ is an edge point on a Kimura edge, $B_{r_2}(q)$ is taken to be the semi-ball that lies in $P$. 
					
					Let $\chi_1, \chi_2$ be the indicatrix functions of $B_{r_1}(p), B_{r_2}(q)$. Instead of deriving the mean value inequalities from local Sobolev inequality in \cite[Theorem 5.2.9]{saloff-coste_2001}, we directly import two mean value inequalities. If $q\in\mathring{P}$, $L$ is elliptic in $B_{r_2}(q)$ and we can apply the sup norm inequality with $m=1$ in \cite[Theorem 6.17]{lieberman1996second}. If $q$ is on a transverse Kimura edge, we can apply 
					\cite[Lemma A.6]{10.1093/amrx/abw002} and the proof in Corollary 4.3. Apart from these two modifications, following the proof in \cite{saloff-coste_2001}, we show that for $t\geq \max(r^2_1,r_2^2)$,
					\begin{gather}\nonumber
						p_t(p,q)\leq\frac{C}{\sqrt{\mu(B_1)\mu(B_2)}}\text{exp}\left[\alpha^2t+\alpha(\phi(p)-\phi(q))+C_0|\alpha|t+|\alpha|(r_1+r_2)+C_1t\right].
					\end{gather}
					Taking $\alpha=-\frac{\phi(p)-\phi(q)}{2t}$, then for $t\geq \epsilon^{-2}\max(r^2_1, r^2_2, \epsilon^4)$,
					\begin{gather}\nonumber
						p_t(p,q)\leq\frac{C}{\sqrt{\mu(B_1)\mu(B_2)}}\text{exp}\left[-\frac{(\phi(p)-\phi(q))^2}{4t}+\left(\frac{C_0}{2}+1\right)|\phi(p)-\phi(q)|+C_1t\right].
					\end{gather}
					
					\paragraph{Proof of \eqref{int_v(t,x)}}
					Having obtained the pointwise upper bound of $p_t(p,q)$, we can estimate $v(t,x)$.
					We choose  $\phi(q)=\phi(p)-C_2x$ with some constant $C$ so that $|\nabla \phi|\leq 1$.
					Since $h|_{H\times [0,A]}\equiv 1$, by definition of $v(t,x)$ \eqref{def:v(t,x)}, we only need to integrate $y$ in $[A,1]$.  We use the pointwise upper bounded above to obtain the following estimate of $v(t,x)$:
					\begin{gather}\label{v(t,x)_pointwise}
						v(t,x)\leq C\cdot\text{exp}\left[-\frac{x^2}{4C_2t}+C_0|x|+C_1t\right],
					\end{gather}
					for some constant $C_0, C_1, C_2>0$.
					Next we estimate $\int\psi'(f)v(t,x)dx$. We use H\"older's inequality with $\frac{1}{r}+\frac{1}{s}=1$ to obtain
					\begin{gather}\label{98}
						\left|\int\psi'(f)v(t,x)dx\right|\leq||\psi'(f)\mu^{1-\frac{1}{p}}||_{\frac{p}{p-1}}\cdot||\mu^{\frac{1}{p}-1} v^{1-\frac{1}{ps}}||_{pr}\cdot ||v^{\frac{1}{ps}}||_{ps}.
					\end{gather}
					First it is easy to see
					\begin{gather}\label{first_estimate}
						||v^{\frac{1}{ps}}||_{ps}=O(e^{\frac{\lambda_0}{ps} t})
					\end{gather}
					has exponential decay with rate $\frac{\lambda_0}{ps}$. 
					We next estimate the first term. Since $\psi(x)\sim x^p$ as $x\to\infty$ and taking into account that $\psi$ is negative only on a finite interval, we have
					\begin{gather}\label{99}
						||\psi'(f)\mu^{1-\frac{1}{p}}||^{\frac{p}{p-1}}_{\frac{p}{p-1}}\leq M_1\int |\psi|\mu\leq M_1\int \psi\mu+M_2.
					\end{gather}
					
					For the second term, notice that $\mu^{-1}$ has asymptotic behavior $O(e^{a|x|})$ for $a>0$, and $v(t,x)$ decays quadratically in $|x|$ by \eqref{v(t,x)_pointwise}. So the second term is intergable while it may have exponential growth with respect to $t$. Indeed after a direct calculation 
					\begin{align*}
						||\mu^{\frac{1}{p}} v^{1-\frac{1}{ps}}||_{pr}&=O\left(\sqrt{t}\cdot \text{exp}\left[4\left(C_0+\frac{a(p-1)}{p-\frac{1}{s}}\right)^2(1-\frac{1}{ps})\right]\right)\\
						&=O\left(\sqrt{t}\cdot \text{exp}\left[4\left(C_0+a\right)^2(1-\frac{1}{ps})\right]\right)
					\end{align*}
					since $\frac{1}{s}<1, p<2$. Regarding \eqref{first_estimate}, for $0<\epsilon<1$, we take $\frac{1}{ps}=1-\epsilon\frac{\lambda_0}{4(C_0+a)^2+\lambda_0}$ so that
					\begin{gather}\label{100}
						||\mu^{\frac{1}{p}-1} v^{1-\frac{1}{ps}}||_{pr}\cdot ||v^{\frac{1}{ps}}||_{ps}=O(e^{(1-\epsilon)\lambda_0)}).
					\end{gather}
					Taken $\eqref{98},\eqref{99}, \eqref{100}$ and an interpolation of arithmetic-geometric mean inequality together, we obtain
					\begin{gather}\label{101}
						\left|\int\psi'(f)v(t,x)dx\right|\leq\epsilon\int\psi\mu+\sqrt{t}e^{(1-\epsilon)\lambda_0 t}.
					\end{gather}
				\end{proof}
				
				\begin{lemma}\label{lem:stratification}
					There exists a stratification as described in \eqref{cover}.
				\end{lemma}
				\begin{proof}
					We first take a finite covering of $H$ by coordinate neighborhoods of the form $R$ based at points of $H$. Let $k_1$ be the minimum of the height (normal direction in the local coordinates) of these rectangles. Then by taking the closure of these rectangles and shrinking the height, $P_{[k_0,k_1]}$ is covered by closed rectangles that satisfy the condition \ref{layer_b}. We can do the similar job at each level of the corners. Specifically, for a corner $p$, the level $L_{\rho(p)}:=\rho=\rho(p)$ is the disjoint union  
					\[L_{\rho(p)}= L_{\rho(p),\max}\bigcup L_{\rho(p),i}\]
					where $L_{\rho(p),\max}$ are the collection of points that attain the local maximum. $ L_{\rho(p),\max}$, $L_{\rho(p),i}$ are both compact and a disjoint union of finite points and the edges out of condition \ref{assp:rhoexistence}.c. We cover $L_{\rho(p),\max}, L_{\rho(p),i}$ by finite open sets of the form $R$ or $T$. Again, by taking the closure and shrinking the height of these open sets, there are $\epsilon_1, \epsilon_2\geq 0$ so that $P_{[\rho(p)-\epsilon_1,\rho(p)+\epsilon_2]}$ is covered by the closed sets that satisfy the condition \ref{layer_b}. Up to now, the remaining parts of $P$ we have not covered are a disjoint unions of level intervals. Suppose $P_{[a,b]}$ is one of level interval, we first cover $P_{[a,b]}$ by finite open rectangles with two sides lie on two $\rho$-levels, then we can cut these rectangles so that $P_{[a,b]}$ is covered by closed rectangles that satisfy the condition \ref{layer_b}. Thus we can cover these level intervals by the desired stratification.
				\end{proof}


\bibliographystyle{siam}
\bibliography{ref1dcase}

\end{document}